\documentclass[a4paper]{amsart}

\usepackage{amsmath}
\usepackage{amssymb}
\usepackage{amsrefs}

\newtheorem{theorem}{Theorem}[section]
\newtheorem*{theorem*}{Theorem}
\theoremstyle{plain}

\newtheorem{corollary}[theorem]{Corollary}

\newtheorem{lemma}[theorem]{Lemma}

\newtheorem{proposition}[theorem]{Proposition}

\theoremstyle{definition}
\newtheorem{definition}[theorem]{Definition}

\newtheorem{remark}[theorem]{Remark}
\newtheorem*{remark*}{Remark}

\newcommand{\CC}{\mathbb{C}}
\newcommand{\FF}{\mathbb{F}}

\newcommand{\RR}{\mathbb{R}}

\newcommand{\frakbb}{\mathfrak{b}}

\newcommand{\frakgg}{\mathfrak{g}}

\newcommand{\calF}{\mathcal{F}}
\newcommand{\calH}{\mathcal{H}}
\newcommand{\calK}{\mathcal{K}}
\newcommand{\calL}{\mathcal{L}}

\newcommand{\ii}{\sqrt{-1}}
\newcommand{\dt}{\left.\frac{d}{dt}\right|_{t=0}}

\DeclareMathOperator{\grad}{grad}

\DeclareMathOperator{\Val}{Val}
\DeclareMathOperator{\vol}{vol}
\DeclareMathOperator{\img}{im}

\DeclareMathOperator{\Klain}{Kl}
\DeclareMathOperator{\spn}{span}

\DeclareMathOperator{\Curv}{Curv}

\DeclareMathOperator{\Area}{Area}
\DeclareMathOperator{\glob}{glob}
\DeclareMathOperator{\Grass}{Gr}
\DeclareMathOperator{\Vector}{Vec}
\DeclareMathOperator{\Angular}{Ang}

\usepackage{hyperref}

\begin{document}
\author{Thomas Wannerer}
\address
	{
		\begin{flushleft}	
			ETH Z\"urich\\
			Departement Mathematik\\
			HG F 28.3\\
			R\"amistrasse 101\\
			8092 Z\"urich, Switzerland \\
		\end{flushleft}
	}	
	
	\email{thomas.wannerer@math.ethz.ch}

\title{The module of unitarily invariant area measures}
\thanks{Supported by DFG grant BE 2484/5-1, FWF grant P22388, and SNF grant 200021-140467}

\begin{abstract}
 The hermitian analog of Aleksandrov's area measures of convex bodies is investigated. A characterization of those area measures which arise as the first variation of unitarily invariant valuations is established. General smooth area measures are shown to form a module over smooth valuations and the module of unitarily invariant area measures is described explicitly.
\end{abstract}

\maketitle

\section{Introduction}

The fundamental result in integral geometry is the principal kinematic formula, which goes back to the work of Blaschke \cite{blaschke55} and was generalized by Chern \cite{chern52} and Federer \cite{federer59}. It states that
\begin{equation}\label{eq_PKF}
\int_{\overline{O(n)}} \chi(K\cap gL)\; dg= \sum_{i+j=n} \binom{n}{i}^{-1} \frac{\omega_i\omega_j }{\omega_n}\mu_i(K) \mu_j(L),
\end{equation}
where $\chi$ denotes the Euler characteristic, $K$ and $L$ are convex bodies in $\RR^n$ (i.e.\ nonempty, compact, convex subsets), $\overline{O(n)}=O(n)\ltimes \RR^n$ is the isometry group of $\RR^n$, $\omega_k$ denotes the volume of the $k$-dimensional euclidean unit ball, and the $\mu_k$ are the intrinsic volumes, see e.g.\ \cite{klain_rota97}. In the linear space $\RR^n$, replacing the intersection in \eqref{eq_PKF} by the vector sum and the Euler characteristic by the $n$-dimensional volume, yields the additive principal kinematic formula 

\begin{equation}\label{eq_APKF}
\int_{O(n)} \vol_n(K + gL)\; dg= \sum_{i+j=n}  \binom{n}{i}^{-1} \frac{\omega_i\omega_j }{\omega_n} \mu_i(K) \mu_j(L),
\end{equation}
which is equivalent to the intersectional principal kinematic formula \eqref{eq_PKF}.

Although already Nijenhuis \cite{nijenhuis74} suspected an underlying algebraic reason for the structure of the principal kinematic formulas, only through the work of Fu \cite{fu06} and Bernig and Fu \cite{bernig_fu06}, heavily based on the fundamental work of Alesker in the theory of valuations \cites{alesker01, alesker03,alesker04,alesker06,alesker06b,alesker07a, alesker_fu08}, the algebraic nature of the principal kinematic formulas was uncovered. Their results provide the tools to obtain explicit kinematic formulas in more general settings. In particular, the orthogonal group may be replaced by any closed subgroup $G\subset O(n)$ acting transitively on the unit sphere, see \cites{bernig09,bernig11,bernig_fu11}. Recently, this algebraic approach was successfully applied by Bernig and Fu \cite{bernig_fu11} to obtain explicit principal kinematic formulas for $G=U(n)$.

There are two ways of localizing the intrinsic volumes: Federer's curvature measures $C_k(K,\;\cdot\;)$ \cite{federer59}, which are measures on $\RR^n$, and Aleksandrov's area measures $S_k(K,\;\cdot\;)$ \cite{aleksandrov37} which are measures on the unit sphere $S(\RR^n)$.
If $K$ is  strictly convex and has a smooth boundary, then curvature and area measures can be expressed as integrals of the elementary symmetric functions in the principal curvatures and principal radii of curvature, respectively. There also exist local versions of the kinematic formulas \eqref{eq_PKF} and \eqref{eq_APKF}. In the latter case, Schneider~\cite{schneider75} proved that
\begin{equation}\label{eq_localadd}\int_{O(n)} S_{n-1} (K+gL, A\cap gB)= \frac{1}{n\omega_n}\sum_{i+j=n-1}\binom{n-1}{i}  S_i(K,A) S_j(L,B)\end{equation}
for all Borel sets $A,B\subset S^{n-1}$.  Very recently, Bernig, Fu, and Solanes \cite{bernig_etal12} established a local version of \eqref{eq_PKF} in hermitian vector spaces (and in fact in all complex spaces forms; see also \cite{abardia_etal12}).

The purpose of this article is to investigate the class of unitarily invariant area measures and to provide the algebraic machinery needed to establish an explicit local version of \eqref{eq_APKF} in hermitian vector spaces. The crucial construction in \cite{bernig_etal12} uses the Alesker product of smooth valuations to turn the space of smooth curvature measures into a module over smooth valuations. Building on this idea, we show that the Bernig-Fu convolution of smooth valuations can be used to define a module structure on the space of smooth area measures. This module structure restricts to unitarily invariant area measures, which are precisely those measures which will appear in the complex version of \eqref{eq_localadd}.

Let us now describe the results of the paper. In Section~\ref{sec_valandarea}, we develop a general theory of smooth area measures. We define the globalization map $\glob$, the first variation map $\delta$, and the centroid map $C$. The first two constructions correspond to those for curvature measures \cites{bernig_fu11,bernig_etal12}; the third one, however, is new and possesses no analog in the theory of curvature measures. Let $\Val^{sm}=\Val^{sm}(V)$ denote the space of smooth, translation-invariant valuations on a euclidean vector space $V$.  We show that the Bernig-Fu convolution can be used to define a module structure on $\Area=\Area(V)$, the vector space of smooth area measures on $V$. We show that the globalization, the first variation and the centroid map are compatible with this module structure.   
 
If $V=\CC^n$, then the module structure can be restricted to unitarily invariant area measures and valuations; we denote these spaces by $\Area^{U(n)}$ and $\Val^{U(n)}$. In Section~\ref{sec_unitarilyinv}, we begin a detailed investigation of unitarily invariant area measures. Our first result on the structure of $\Area^{U(n)}$ relates the kernel of the centroid map $C$ with the image of the first variation map $\delta$. The first variation map $\delta \colon \Val^{sm}\rightarrow \Area$ is uniquely determined by the property that
$$\dt \phi(K +tL)=\int_{S(V)} h_L \; d(\delta \phi(K))$$
for all convex bodies $K$ and $L$. Here $S(V)$ denotes the unit sphere of $V$ and $h_L(u)=\sup_{x\in L} \langle u,x\rangle $ the support function of $L$. Given an area measure $\Psi$ and a convex body $K$, the centroid map $C$ yields the centroid of the measure $\Psi(K)$ (see Section~\ref{sec_valandarea} for the precise definitions). 

\begin{theorem*}
Let $\Psi\in \Area^{U(n)}$. Then 
$$C(\Psi)=0\quad \text{if and only if}\quad \Psi=\delta\phi $$
for some $\phi\in\Val^{U(n)}$.
\end{theorem*}

The centroid map sends area measures on $V$ to valuations on $V$ with values in $V$. If the area measure is unitarily invariant, then the resulting $\CC^n$-valued valuation is unitarily equivariant. Using the explicit description of the isotypical decomposition of $\Val^{sm}$ under the action of the orthogonal group $O(V)$, we determine the dimension of $\Vector^{U(n)}$, the vector space of unitarily equivariant, translation-invariant, continuous, $\CC^n$-valued valuations. We denote by $\Vector^{U(n)}_k\subset\Vector^{U(n)}$ the subspace of $k$-homogeneous valuations. Observe that $\Vector^{U(n)}_k$ is a complex vector space. 

\begin{theorem*}
 $$\dim_\CC \Vector^{U(n)}_k=\dim_\RR \Val_k^{U(n)}-1.$$
\end{theorem*}

As an application of this result, we obtain a new characterization of the Steiner point map in hermitian vector spaces.

In Section~\ref{sec_module}, we explicitly determine the module structure of $\Area^{U(n)}$. The main result is the following theorem. 
Recall that equipped with the Alesker product $\Val^{U(n)}$ is an algebra generated by two special elements $s$ and $t$, see \cite{fu06}. In the theorem below we consider  $\Val^{U(n)}
\oplus \Val^{U(n)}$ as a $\Val^{U(n)}$-module under the diagonal action.

\begin{theorem*}
The module of unitarily invariant area measures is generated by two elements. More precisely,
$$\Area^{U(n)}\cong (\Val^{U(n)}\oplus \Val^{U(n)})/I_n,$$
where $I_n$ is the submodule generated by the following pairs of valuations
$$(p_n,-q_{n-1})\quad \text{and}\quad (0,p_n),$$  
which are determined by the Taylor series expansions
$$\frac{1}{1+tx+sx^2}=\sum_{k=0}^\infty p_k(s,t) x^k$$ 
and 
$$ -\frac{1}{(1+tx+sx^2)^2}=\sum_{k=0}^\infty q_k(s,t) x^k.$$
\end{theorem*}

We note that the above theorem fits beautifully with Fu's description of the algebra of unitarily invariant valuations \cite{fu06} (see also Theorem~\ref{thm_fu}).

\section{Valuations and area measures} \label{sec_valandarea} 

\subsection{Definitions and results from valuation theory} Throughout this article $V$ will denote a finite-dimensional euclidean vector space equipped 
with the inner product $\langle\; ,\; \rangle$ and the norm $|\cdot|$. We put $\Grass_k=\Grass_k(V)$ for the Grassmannian of $k$-dimensional, linear subspaces of $V$. 
 We denote by $\calK(V)$ the space of convex bodies, i.e.\ nonempty, compact, convex subsets of $V$, equipped with the Hausdorff metric, and we write $\calK^{sm}(V)\subset\calK(V)$ for the subset of convex bodies with nonempty interior whose boundary is an embedded smooth submanifold of $V$ and for which all principal curvatures are positive. We put
$$\omega_k=\frac{\pi^{\frac{k}{2}}}{\Gamma(\frac{k}{2}+1)}$$
for the volume of the $k$-dimensional euclidean unit ball.

Let $A$ be an abelian semigroup. A (convex) valuation on $V$ is a map $\phi:\calK(V)\rightarrow A$ such that
$$\phi(K\cup L)+ \phi(K\cap L)=\phi(K)+\phi(L),$$
whenever $K, L, K\cup L\in\calK(V)$. 
If $A=\RR$, we speak of scalar valuations. The simplest examples of scalar valuations are given by the Euler characteristic $\chi$ and the Lebesgue measure $\vol_n$, $n=\dim V$. For the purposes of this article it is sufficient to consider valuations with values in a (finite-dimensional) vector space and for these valuations we have a particularly rich theory at our disposal. We note that also valuations with values in $A=\calK(W)$, where $W$ is some vector space, have been extensively studied, in particular in connection with affine isoperimetric inequalities, see e.g.\ \cites{abardia_bernig11,abardia12,ludwig02,ludwig03,haberl11,haberl_schuster09, ludwig06,ludwig10,lutwak_etal00, lutwak_etal10,schuster10,schuster08,wannerer12}.

After these basic definitions, we recall now some definitions and results from the theory of scalar valuations.
 For an overview of the subject the reader is advised to consult the survey articles \cites{bernig12a, alesker07b}.
 A good introduction to the classical theory of valuations is the book \cite{klain_rota97}. For recent important results see \cite{alesker99} and \cite{ludwig_reitzner10}. 
We denote by $\Val=\Val(V)$ the space of translation-invariant, continuous, scalar valuations and by $\Val^{sm}\subset\Val$ the dense subspace of smooth valuations. Recall from \cite{alesker03}
 that  a valuation $\phi\in\Val$ is called smooth if 
$g\mapsto g\cdot \phi$ is a smooth map from $GL(V)$ to $\Val$, where $g\cdot \phi(K) = \phi(g^{-1} K)$ for every $K\in \calK(V)$ and $GL(V)$ denotes the general linear group.
In the following every valuation is tacitly assumed to be translation-invariant and at least continuous.
 A family of examples of smooth valuations is given by 

$$\phi_A(K)=\vol_n(A+K), \qquad A\in\calK^{sm}.$$ 
A valuation $\phi$ is called homogeneous of degree $k$ if $\phi(t K)=t^k \phi(K)$ for $t>0$. We call $\phi$ even if $\phi(-K)=\phi(K)$; $\phi$ is called odd if $\phi(-K)=-\phi(K)$. 

Let $\Val_k^+\subset\Val$ denote the subspace of $k$-homogeneous and even valuations. It is well-known \cite{klain00} that the restriction of $\phi\in \Val_k^+$ to a $k$-dimensional subspace $E\in\Grass_k$ is proportional to the $k$-dimensional volume. Denoting this proportionality factor by $\Klain_\phi(E)$, we obtain a function on the Grassmannian called the Klain function of $\phi$. A theorem of Klain \cite{klain00} states that the map which sends $\phi$ to its Klain function $\Klain_\phi$ is injective.

One of the striking features of smooth valuations is that they exhibit a rich algebraic structure.  We start our discussion of the various algebraic operations on valuations with Alesker's Fourier transform  $\FF:\Val^{sm}\rightarrow\Val^{sm}$ (see \cite{alesker11}). For the sake of brevity we will sometimes simply write $\widehat\phi$ instead of $\FF\phi $.
In this article we only use the Fourier transform for even valuations and in this case it is uniquely determined by the equation
$$\Klain_{\widehat\phi}(E)=\Klain_\phi(E^\perp),\qquad E\in \Grass_k.$$
In particular, we see that $\FF$ is an involution on the space of even valuations. Consider for example the $k$-th intrinsic volume $\mu_k\in\Val_k^{sm}$. Since the $k$-th intrinsic volume of a $k$-dimensional convex body equals precisely its $k$-dimensional volume, we have $\Klain_{\mu_k}=1$ and therefore
\begin{equation}\label{eq_fintrinsic}
\widehat\mu_k=\mu_{n-k}.
\end{equation}
Bernig and Fu introduced in \cite{bernig_fu06} a continuous, commutative convolution product on $\Val^{sm}$. The convolution possesses---and is in fact characterized by---the property that for any valuation $\psi\in\Val^{sm}$ and $A\in \calK^{sm}$
\begin{equation}\label{eq_convprod}
\phi_A*\psi=\psi(\; \cdot\; + A).
\end{equation}
As an important example let us compute $\mu_{n-1}*\psi$. We have
$$\mu_{n-1}(K)=\frac{1}{2}\dt \vol_n(K+tB(V)),$$
where $B(V)$ denotes the unit ball of the euclidean vector space $V$. By the continuity of the convolution product and by \eqref{eq_convprod} we obtain
\begin{equation}\label{eq_convintrinsic}
\mu_{n-1}*\psi=\frac{1}{2}\dt \psi(\;\cdot\;+tB(V))
\end{equation}
whenever $\psi\in\Val^{sm}$. 
The convolution product is related to the Alesker product \cite{alesker04} via the Fourier transform
\begin{equation}\label{eq_fourierhom}
\FF(\phi\cdot \psi)=\FF\phi * \FF\psi,
\end{equation}
see \cites{bernig_fu06, alesker11}.

Let $S(V)$ denote the unit sphere of $V$ and write $SV=V\times S(V)$ for the sphere bundle of $V$. Since $SV$ is a cartesian product, there are two natural projections $\pi_1: SV\rightarrow V$ and $\pi_2: SV\rightarrow S(V)$. 
It is well-known that each translation-invariant, smooth differential form $\omega\in\Omega^{n-1}(SV)$, gives rise to a smooth, translation-invariant valuation via integration,
$$K\mapsto\int_{N(K)}\omega.$$
Here $N(K)$ denotes the normal cycle of $K\in\calK(V)$, see \cites{fu94, alesker_fu08}. Moreover, every smooth, translation-invariant valuation can be written in the form 
$$c\vol_n +\int_{N(K)}\omega$$
with some constant $c\in \RR$ and some $\omega$ as above, see \cite{alesker06}. 

We denote by $\Omega^{n-1}(SV)^{tr}\subset \Omega^{n-1}(SV)$ the subspace of translation-invariant forms.
 The kernel of the map $\Omega^{n-1}(SV)^{tr}\rightarrow \Val^{sm}$ given by integration with respect to the normal cycle was determined by Bernig and Br\"ocker \cite{bernig_broecker07} using the Rumin differential operator. The Rumin differential operator \cite{rumin94} is defined on a general contact manifold, but for our purposes it is sufficient to consider it only in the special case of the sphere bundle (we refer to \cite{blair10} for all notions from contact geometry). Let $\alpha$ denote the canonical contact form on the sphere bundle $SV$. 
The Rumin differential $D: \Omega^{n-1}(SV)\rightarrow\Omega^{n}(SV)$ is a second order differential operator given by
$$D\omega=d(\omega + \alpha\wedge\xi),$$
where $\xi \in\Omega^{n-2}(SV)$ is chosen such that $d\omega+d\alpha \wedge \xi=0$ when restricted to the contact plane. In particular, $D\omega$ is a multiple of $\alpha$.

\begin{theorem}[Bernig and Br\"ocker \cite{bernig_broecker07}]\label{thm_kernel}
Suppose $\omega\in \Omega^{n-1}(SV)^{tr}$. The valuation
$$\phi(K)=\int_{N(K)}\omega$$
is the zero valuation if and only if $D\omega=0$ and $\phi(\{v\})=0$ for some point $v\in V$.
\end{theorem}

In particular, we see that 
$$\int_{N(\cdot)}\omega=0,$$
whenever $\omega$ is a multiple of $\alpha$ or $d\alpha$. Finally, we denote by $T$ the Reeb vector field on $SV$; it is uniquely determined by
$$i_T\alpha=1\qquad \text{and} \qquad \calL_T\alpha=0.$$

\subsection{First variation and area measures}

Let $\phi\in\Val$ be a valuation. We say that a signed Borel measure $m$ on the unit sphere is the first variation of $\phi$ at $K$ if 
$$\dt \phi(K+tL)=\int_{S(V)} h_L\;dm$$
for every $L\in\calK(\RR^n)$. Here $h_L(u)=\sup_{x\in L} \langle u,x \rangle$ denotes the support function of $L$. The case $\phi=\vol_n$ is  classical and, in fact, the first variation of the volume at $K$ coincides precisely with \emph{the} area measure of $K$,
\begin{equation}\label{eq_variationvol}\dt \vol_n(K+tL)=\int_{S(V)} h_L\; dS_{n-1}(K),\end{equation}
see e.g.\ \cite{schneider_book}*{p.~203}.
To set the stage for our definition of general smooth area measures, we first consider measures on the unit sphere which arise as the first variation of translation-invariant, smooth valuations.

\begin{proposition}\label{thm_firstvar}
Suppose $\phi\in\Val^{sm}$ and $K\in\calK(V)$. Then there exists a unique, signed Borel measure $\delta\phi(K)$ on $S(V)$, called the first variation of $\phi$ at $K$, such that
\begin{equation}
\label{eq_firstvar}\dt \phi(K+tL)=\int_{S(V)} h_L\; d(\delta\phi(K))\end{equation}
for every $L\in \calK(V)$.
 
\end{proposition}

\begin{remark}
\begin{enumerate}
\item The case where $\phi=\mu_k$ is an intrinsic volume is classical; in fact,

$$\delta\mu_k(K)=\frac{1}{\omega_{n-k}} \binom{n-1}{k-1} S_{k-1}(K),$$ 
where $S_k(K)$ denotes the \emph{$k$-th area measure} of $K$, see e.g.\ \cite{schneider_book}*{p. 203}.

\item In \cite{bernig_fu11} the first variation of a valuation was introduced as a curvature measure, not as an area measure (see below for the definitions). This is more suitable if the first variation is considered with respect to the deformation of $K$ under the flow of a vector field on $V$.
\end{enumerate}
\end{remark} 

\begin{proof}[Proof of Proposition~\ref{thm_firstvar}]
It follows from a well-known result of McMullen \cite{mcmullen77} that $\phi(K+tL)$, $t\geq 0$, is a polynomial in $t$; thus the left hand side of
 \eqref{eq_firstvar} is well-defined and continuous in $K$ and $L$. Uniqueness follows from the fact that the span of differences of support
 functions is a dense subspace of all continuous functions on the unit sphere, see e.g.~\cite{schneider_book}*{Lemma 1.7.9}.  It remains to prove existence. Since $\phi$ is a translation-invariant, smooth valuation, there exists a constant $c\in\RR$ and a translation-invariant, smooth differential form $\omega\in\Omega^{n-1}(SV)$ such that 
$$\phi(K)=c\vol_n(K)+\int_{N(K)} \omega.$$
By \eqref{eq_variationvol} and the fact that the normal cycle vanishes on multiples of $\alpha$, we may assume without loss of generality that $c=0$ and $D\omega= d\omega$.
 This assumption implies in particular that $d\omega$ is a multiple of $\alpha$.
Fix now two convex bodies $K\in\calK$ and $L\in \calK^{sm}$ and for each $t\in\RR$ define a diffeomorphism $F_t: SV\rightarrow SV$ by
$$(x,v)\mapsto (x+t \nabla h_L(v), v).$$
Since the boundary of $L$ can be expressed as $\{\nabla h_L(v): v\in S(V)\}$, it is easy to check that $F_t(N(K))=N(K+tL)$. Furthermore, note that $F_{s+t}=F_s\circ F_t$. Let $X$ denote the vector field on $SV$ generated by the one-parameter subgroup of diffeomorphisms $t\mapsto F_t$. 
We compute
\begin{align*}
\dt \phi(K+tL)&= \dt \int_{F_t(N(K))}\omega = \int_{N(K)}\dt F_t^*\omega  \\
	&=  \int_{N(K)}\calL_X\omega = \int_{N(K)}\alpha(X)\wedge i_T d\omega  \\
	&=  \int_{N(K)}\pi_2^*h_L\wedge	i_T d\omega,		
\end{align*}
where we have used $\calL_X\omega=d(i_X\omega)+i_Xd\omega$, $\partial N(K)=0$, $d\omega=\alpha\wedge i_T d\omega$, the fact that the normal cycle vanishes on multiples of $\alpha$, and  
$$\alpha(X)_{(x,v)}=\langle\nabla h_L(v),v \rangle=h_L(v).$$
By continuity, we obtain 
\begin{equation}\label{eq_dtmu}
\dt \phi(K+tL)=  \int_{N(K)}\pi_2^*h_L\wedge	i_T d\omega= \int_{S(V)} h_L \; d(\delta\phi(K))
\end{equation}
for general convex bodies $K$ and $L$, where the Borel measure $\delta\phi(K)$ is given explicitly by
\begin{equation}
 \label{eq_varmeas} \delta\phi(K)= \pi_{2*}( N(K)\: \llcorner\: i_T d\omega).
\end{equation}
This completes the proof of the proposition.
\end{proof}

\begin{corollary}\label{cor_convLie} If $\phi=\int_{N(\cdot)} \omega$, then
$$2\mu_{n-1}*\phi=\int_{N(\cdot)}\calL_T\omega.$$

\end{corollary}
\begin{proof}
Since $h_{B(V)}=1$ and $\partial N(K)=0$, this is an immediate consequence of \eqref{eq_convintrinsic} and  \eqref{eq_dtmu}.
\end{proof}

We see from \eqref{eq_varmeas}, that the first variation measure of a smooth valuation is given by integration of a translation-invariant, smooth $(n-1)$-form over a part of the normal cycle of a convex body. This motivates the following definition of general smooth area measures.

\begin{definition}[Smooth area measures]

The vector space $\Area=\Area(V)$ of (smooth) area  measures on $V$ is given by all expressions of the form
$$\Psi(K,A)=\int_{N(K)\cap\pi^{-1}_2(A)}\omega.$$
Here $K\in\calK(V)$ is a convex body,   $\omega\in\Omega^{n-1}(SV)$ a translation-invariant, smooth $(n-1)$-form, $A\subset S(V)$ a Borel set, and $\pi_2:SV\rightarrow S(V)$ the 
canonical projection. Furthermore, we denote by $\Area_k\subset \Area$ the subspace of area measures given by differential forms which are homogeneous of degree $k$. Here we call
 $\omega\in \Omega^{n-1}(SV)^{tr}$ homogeneous of degree $k$ if $m_t^* \omega =t^k\omega$, $t>0$, where $m_t\colon SV\rightarrow SV$ denotes multiplication by $t$ in the
 first component $m_t(x,v)=(tx,v)$.

\end{definition}

\noindent\textbf{Notation.}
Given an area measure $\Psi\in \Area$, $K\in\calK(V)$, and a bounded Borel function $f\colon S(V)\rightarrow \RR$,  we will denote integration with respect to the measure $\Psi(K)=\Psi(K,\;\cdot\;)$ by
$$\int_{S(V)} f(u) \; d\Psi(K,u).$$

\begin{remark}
\begin{enumerate}
\item Note that a smooth area measure is by definition not a measure, but a map which assigns to every convex body a measure on the unit sphere. It follows from \cite{alesker_fu08}*{Corollary 2.1.10} that $K\mapsto \Psi(K)$ is a valuation with values in the vector space of signed Borel measures on the unit sphere.
\item Since the exterior powers satisfy $\Lambda^m(V\times W)\cong\bigoplus_{k=0}^m \Lambda^k V\otimes \Lambda^{m-k}W$ whenever $V$ and $W$ are vector spaces, we clearly have 
$$\Area=\bigoplus_{k=0}^{n-1} \Area_k.$$
Moreover, $\Psi\in\Area_k$ if and only if $\Psi(t K)=t^k\Psi(K)$ whenever $t> 0$ and $K\in\calK(V)$. 
\item Observe that $\delta\phi\in \Area$ whenever $\phi\in\Val^{sm}$. In particular, the classical area measures $S_0,\ldots, S_{n-1}$ are smooth area measures. 
\end{enumerate}
\end{remark}

An in a certain sense dual notion to smooth area measures are \emph{smooth curvature measures}. These are maps which send convex bodies to signed Borel measures on $V$,
$$\Phi(K,A)=\int_{N(K)\cap\pi^{-1}_1(A)}\omega, \qquad A\subset V,$$ 
see e.g.\ \cites{bernig_fu11,bernig_etal12, federer59, fu90, fu94}. The map $(K,A)\mapsto \vol_n(K\cap A)$ is also considered to be a curvature measure. We denote by $\Curv=\Curv(V)$ the vector space of all curvature measures on $V$. We explore the relations between area and curvature measures in Subsection~\ref{sec_areacurv}.

From the definition of area measures we see that they can be considered as a special way of mapping convex bodies to measures on the unit sphere. This suggests to consider the following two basic operations: (1) evaluating each measure on the whole unit sphere and (2) computing the centroid of each measure. This is the content of the next definition.

We denote by $\Vector=\Vector(V)$ the vector space of continuous, translation-invariant valuations on $V$ with values in $V$, i.e. $\Vector(V)\cong\Val\otimes V$. The subspace of smooth valuations $\Vector^{sm}\subset\Vector$ is given by those elements which can represented by integration of a smooth, translation-invariant $(n-1)$-form on $SV$ with values in $V$ over the normal cycle.

\begin{remark} Equivalently, we could have defined $\Vector^{sm}$ as the subspace of smooth vectors of the natural $GL(V)$-representation on $\Vector(V)\cong\Val\otimes V$. Indeed, since $(\Val\otimes V)^{sm}=\Val^{sm}\otimes V$ in terms of smooth vectors (see e.g.\ \cite{alesker04}*{Lemma 1.5} for a proof), the subspace of smooth vectors coincides with the subspace of valuations which can be represented by integration of a smooth differential form with values in $V$.
\end{remark}

\begin{definition}
We denote by $\glob: \Area\rightarrow \Val^{sm}$ the \emph{globalization map} 
$$\glob(\Psi)=\Psi(\;\cdot\;,S(V)).$$
The map
$C:\Area\rightarrow \Vector^{sm}$ defined by
$$C(\Psi)=\int_{S(V)}u\; d\Psi(\;\cdot \;,u) $$
is called the \emph{centroid map}.
\end{definition}

The following lemma establishes a first connection between the first variation of a valuation and the centroid map.

\begin{lemma}\label{lem_centroids}
Let $\Psi\in\Area$. If there exists $\phi\in\Val^{sm}$ such that $\Psi=\delta \phi $, then $C(\Psi)=0$.  
\end{lemma}
\begin{proof}
Since $\phi$ is translation-invariant and $h_{\{v\}}(u)=\langle u,v\rangle$, it follows from \eqref{eq_firstvar} that
$$0=\int_{S(V)} \langle u,v \rangle \; d\Psi(K,u)=\langle C(\Psi)(K),v \rangle$$
for each $v\in V$. Thus, $C(\Psi)(K)=0$ for every $K\in\calK(V)$.
\end{proof}

\subsection{Modules over \texorpdfstring{$\Val^{sm}$}{Valsm}}

Both the Alesker product and the Bernig-Fu convolution product turn the vector space of smooth valuations into an algebra with unit satisfying Poincar\'{e} duality. It was shown in \cite{bernig_etal12} that one of these operations, namely the Alesker product, can be used to turn the vector space of smooth curvature measures into a module over smooth valuations. Building on this idea, we show that in the case of area measures one can use the convolution product to turn the vector space of smooth area measures into a module over smooth valuations. This module structure is compatible with the first variation map $\delta$, the globalization map $\glob$, and the centroid map $C$.

Let us start by recalling the description of the convolution of valuations in terms of differential forms, see \cite{bernig_fu06}.
 Suppose we are given two smooth, translation-invariant valuations $\phi, \psi\in \Val^{sm}$,
$$\phi(K)=\int_{N(K)} \beta \quad\text{and}\quad \psi(K)=\int_{N(K)} \gamma.$$ 
Since the normal cycle vanishes on multiples of $\alpha$, we may assume that $D\beta=d\beta$ and $D\gamma=d\gamma$. In terms of $\beta $ and $\gamma$, the convolution $\phi*\psi$ is given by 
\begin{equation}\label{eq_defconvprod}
\phi *\psi=\int_{N(\cdot)} *_1^{-1}(*_1\beta\wedge *_1d\gamma)
\end{equation}
Here $*_1 $ is a linear operator on $\Omega^*(SV)^{tr}$ which is uniquely determined by the relation
$$*_1(\pi_1^*\gamma_1\wedge\pi_2^*\gamma_2)=(-1)^{\binom{n-\deg \gamma_1}{2}} \pi_1^*(*_V\gamma_1)\wedge \pi_2^*\gamma_2,$$
where $\pi_1:SV\rightarrow V$ and $\pi_2:SV\rightarrow S(V)$ denote the natural projections, $\gamma_1\in\Omega^*(V)$, $\gamma_2\in\Omega^*(S(V))$, and $*_V$ is the Hodge 
star operator on $\Omega^*(V)$.

\begin{definition} Whenever $f$ is a smooth function on the unit sphere and $\Psi\in\Area$ we define a smooth, translation-invariant valuation $\Psi_f\in\Val^{sm}$ by 
$$\Psi_f(K)=\int_{S(V)}f(u)\; d\Psi(K,u).$$
\end{definition}
Observe that $\Psi_f$ is indeed a smooth, translation-invariant valuation, since it can obviously be represented by a smooth, translation-invariant differential form.

\begin{proposition}
For each $\phi\in \Val^{sm}$ and $\Psi\in \Area $ there exists a unique area measure $\phi*\Psi\in \Area$ such that
\begin{equation}\label{eq_conv}
(\phi*\Psi)_f=\phi*\Psi_f
\end{equation}
for every  $f\in C^\infty(S(V))$.
\end{proposition}
\begin{proof}
Since uniqueness follows immediately from \eqref{eq_conv}, we only prove existence. To this end fix $\phi\in\Val^{sm}$ and $\Psi\in \Area$, say $\phi$ and $\Psi$ are given by
$$\Psi(K,A)=\int_{N(K)\cap\pi^{-1}_2(A)}\omega\quad\text{and}\quad \phi(K)=\int_{N(K)} \beta.$$ 
By Theorem~\ref{thm_kernel}, we may assume that $D\beta=d\beta$. 

Since $\Psi_f$ is a smooth, translation-invariant valuation, the convolution product $\phi*\Psi_f$ is well-defined. 
Let $\xi\in\Omega^{n-2}(SV)$ be such that $D(\pi_2^* f\wedge\omega)=d(\pi_2^*f\wedge\omega+\alpha\wedge \xi)$.
By the definition of the convolution product \eqref{eq_defconvprod}, we have 
\begin{align*}
 \phi*\Psi_f&=\int_{N(\cdot)} *_1^{-1}(*_1(\pi_2^*f\wedge\omega+\alpha\wedge\xi)\wedge *_1d\beta)\\
&= \int_{N(\cdot)} \pi_2^*f \wedge *_1^{-1}(*_1\omega\wedge *_1d\beta)+ \int_{N(\cdot)} *_1^{-1}(*_1(\alpha\wedge\xi)\wedge *_1d\beta)\\
&=\int_{N(\cdot)} \pi_2^*f\wedge *_1^{-1}(*_1\omega\wedge *_1d\beta), 
\end{align*}
where the last equality holds because $d\beta=D\beta$ is a multiple of $\alpha$ and hence by the properties of the Hodge star operator also $*_1^{-1}(*_1(\alpha\wedge\xi)\wedge *_1d\beta)$ is a multiple of $\alpha$ and the normal cycle vanishes on multiples of $\alpha$. If we define now $\phi*\Psi\in\Area$ by
\begin{equation}\label{eq_moduledef}
\phi*\Psi(K,A)= \int_{N(K)\cap\pi^{-1}_2(A)} *_1^{-1}(*_1\omega\wedge *_1d\beta),
\end{equation}
we obtain \eqref{eq_conv}.
\end{proof}

We equip the vector space of all smooth area measures with the quotient topology which is induced by the integration map $\Omega^{n-1}(SV)^{tr}\rightarrow \Area(V)$.
\begin{remark}
The topology on $\Area$ is Fr\'{e}chet. Indeed, the kernel of the integration map consists precisely of forms which are multiples of $\alpha$ and $d\alpha$, see \cite{fu11}*{Proposition 3.6}, and is therefore closed. This implies that the quotient topology is Fr\'{e}chet. We will, however, not use this fact.
\end{remark}

Recall that for $A\in\calK^{sm}$ the valuation $\phi_A\in\Val^{sm}$ is defined by 
$$\phi_A=\vol_n(\;\cdot\;+A).$$

\begin{theorem}
The space $\Area$ of all smooth area measures carries the structure of a module over $\Val^{sm}$ such that the action of $\Val^{sm}$ on $\Area$ is continuous and is uniquely determined by the property that
\begin{equation}\label{eq_modcharacterization}
 \phi_A *\Psi= \Psi(\;\cdot\;+ A)
\end{equation}
whenever $A\in \calK^{sm}$ and $\Psi\in\Area$. 
\end{theorem}
\begin{proof}
Using the fact that the convolution product is associative and bilinear and \eqref{eq_conv}, it is easy to check that $(\phi,\Psi)\mapsto \phi*\Psi$  defines a module structure on $\Area$. 
 Since both $\Val^{sm}$ and $\Area $ are quotients of $\Omega^{n-1}(SV)^{tr}$, we see from \eqref{eq_moduledef} that $(\phi,\Psi) \mapsto \phi*\Psi$ is continuous. 
To prove \eqref{eq_modcharacterization}, observe that \eqref{eq_conv} and \eqref{eq_convprod} imply that
$$(\phi_A*\Psi)_f=\phi_A*\Psi_f=\Psi_f(\;\cdot\;+ A)=\left(\Psi(\;\cdot\;+A)\right)_f,$$
whenever $f\in C^\infty(S(V))$.  Since the linear span of the valuations $\phi_A$ is dense in $\Val^{sm}$ by Alesker's irreducibility theorem \cite{alesker01}, we conclude that equation \eqref{eq_modcharacterization} determines the module structure uniquely.
\end{proof}

\begin{lemma}\label{lem_convintrinsic}
If $\Psi\in \Area$ is given by $\Psi(K,A)=\int_{N(K)\cap \pi_2^{-1}(A)} \omega$, then
$$2\mu_{n-1}*\Psi (K,A)=\int_{N(K)\cap \pi_2^{-1}(A)} \calL_T\omega.$$
\end{lemma}
\begin{proof}
Fix $f\in C^\infty(S(V))$. Using \eqref{eq_conv}, Corollary~\ref{cor_convLie}, and  $\calL_T(\pi_2^* f)=0$, we obtain
$$(2\mu_{n-1}*\Psi)_f=2\mu_{n-1}*\Psi_f=\int_{N(\cdot)} \calL_T(\pi_2^*f\wedge \omega)=\int_{N(\cdot)}\pi_2^*f\wedge \calL_T \omega.$$

\end{proof}

We define an action of $\Val^{sm}$ on $\Vector^{sm}(V)\cong\Val^{sm}\otimes V$ by
$$\phi*(\psi\otimes v)=(\phi*\psi)\otimes v,$$
where $\phi,\psi\in\Val^{sm}$ and $v\in V$. In other words, after a choice of coordinates a scalar valuation acts on a vector valuation componentwise.

\begin{proposition}
The centroid map $C:\Area\rightarrow \Vector^{sm}$, the first variation map $\delta: \Val^{sm}\rightarrow \Area$, and the globalization map $\glob:\Area\rightarrow\Val^{sm}$ are  $\Val^{sm}$-module homomorphisms. Furthermore, 
\begin{equation}\label{eq_deltaconv}\delta(\phi)= \phi* S_{n-1}.\end{equation}
\end{proposition}
\begin{proof}
Since $\glob(\Psi)=\Psi_f$ with $f=1$, equality (\ref{eq_conv}) yields
$$\glob(\phi*\Psi)=\phi*\glob(\Psi).$$
Hence, $\glob$ is a $\Val^{sm}$-module homomorphism.
Fix $\phi\in\Val^{sm}$, $\Psi\in\Area$, and let $\xi\in V^*$ be a linear functional. Using (\ref{eq_conv}), we obtain
\begin{align*}
	\xi(C(\phi*\Psi)) &= (\phi*\Psi)_\xi =\phi* \Psi_\xi\\
&= \phi * \xi(C(\Psi))\\ &=\xi(\phi * C(\Psi)).
\end{align*}
Thus, $C(\phi*\Psi)=\phi*C(\Psi)$.

Since $\delta:\Val^{sm}\rightarrow \Area$ is linear and continuous, it suffices to prove that
\begin{equation}\label{eq_deltahom}\delta(\phi_A*\psi)=\phi_A*\delta(\psi).\end{equation}
Fix $K,L\in \calK(V)$. Using the definition of the first variation and \eqref{eq_convprod}, we compute

\begin{align*}
\int_{S(V)} h_L\; \delta(\phi_A*\psi)(K) & = \dt \phi_A*\psi(K+tL)= \dt \psi(K+A+tL)\\
& =   \int_{S(V)} h_L\; \delta(\psi)(K+A)\\
& =   \int_{S(V)} h_L\; \phi_A*\delta(\psi)(K).
\end{align*}
Since differences of support functions lie dense in the space of continuous functions on the unit sphere, we obtain \eqref{eq_deltahom}. Relation \eqref{eq_deltaconv} follows now from
$$\delta(\phi)=\delta(\phi*\vol_n)=\phi*\delta(\vol_n)=\phi * S_{n-1}.$$
\end{proof}

\subsection{Angular area measures}

It is a well-known fact that the $k$-th intrinsic volume of a polytope $P\in \calK(V)$ is given by
$$\mu_k(P)= \sum_{F\in\calF_k(P)} \angle(F,P) \vol_k(F).$$
Here $\calF_k(P)$ denotes the set of $k$-dimensional faces of $P$ and $ \angle(F,P) $ denotes the normalized external angle of $P$ at its face $F$,  see e.g.\ \cite{schneider_book}*{p. 100}. A corresponding formula holds for the classical area measures $S_k$.

\begin{definition} We define $\Delta_k\in \Area$ by
$n \omega_{n-k} \Delta_k:=\binom{n}{k} S_k$.
\end{definition}

With this renormalization we have
$$\Delta_k(P,A)= \sum_{F\in\calF_k(P)} \frac{\calH^{n-1-k}(N(F,P)\cap A)}{(n-k)\omega_{n-k}} \vol_k(F),$$
where $N(F,P)$ denotes the normal cone of $P$ at its face $F$ and $\calH^d$ denotes the $d$-dimensional Hausdorff measure.
In particular, we see that
  $$\glob(\Delta_k)=\mu_k.$$ 

These considerations lead us to the following definition.

\begin{definition}
A smooth area measure $\Psi\in \Area$ is called \emph{angular} if for every polytope $P$
$$\Psi(P,A)= \sum_{k=0}^{n-1}\sum_{F\in\calF_k(P)} c_\Psi(\bar F) \frac{\calH^{n-1-k}(N(F,P)\cap A)}{(n-k)\omega_{n-k}} \vol_k(F),$$
where the number $c_\Psi(\bar F)$ depends only on $\bar F$, the unique translate of the affine span of $F$ which contains the origin. 
The space of angular area measures is denoted by $\Angular=\Angular(V)$.
\end{definition}

\begin{remark}
\begin{enumerate}
\item If $\Psi\in \Area_k$ is angular, then $\psi=\glob(\Psi)\in\Val^{sm}$ is even and homogeneous of degree $k$. In particular, the Klain function of $\psi$ coincides with $c_\Psi$. The converse, however, is false: There exists $\psi\in \Val^{sm}$, even and of degree $k$, such that there exists no angular area measure $\Psi$ with $\glob(\Psi)=\psi$, see \cite{parapatits_wannerer12}.

\item For curvature measures there exists a similar notion of angularity which was first introduced by Bernig, Fu, and Solanes in \cite{bernig_etal12}.
\end{enumerate}
\end{remark}

\begin{theorem}\label{thm_angularity}
Suppose $\Psi\in\Area$ is angular. Then  $C(\Psi)=0$ if and only if $\Psi$ is a linear combination of the $\Delta_k$.
\end{theorem}
\begin{proof}
 Suppose  $C( \Psi)=0$. Without loss of generality we may assume $\Psi\in \Area_k$.  Let $T$ be a $(k+1)$-dimensional simplex having one vertex at the origin and let $U$ be the smallest linear subspace containing $T$.
 We denote by $F_0,\ldots,F_{k+1}$ the facets of $T$ and by $u_0,\ldots,u_{k+1}$ the facet unit normals which lie in $U$. Since 
$$\int_{S(V)\cap N(T,F_i)}u \; d\calH^{n-k-1}(u)=\omega_{n-k-1} u_i,$$
we obtain
$$
0= C(\Psi)(T)= \omega_{n-k-1} \sum_{i=0}^{k+1} c_\Psi( \bar F_i) \vol_k(F_i)u_i.
$$
Without loss of generality we may assume that $c_\Psi( \bar F_0)\neq 0$. Then from
$$ \vol_k(F_0) u_0=-\sum_{i=1}^{k+1}\frac{c_\Psi( \bar F_i)}{c_\Psi(\bar F_0)} \vol_k(F_i) u_i$$
together with the linear independence of $u_1,\ldots,u_{k+1}$ and the fact that
$$\vol_k(F_0) u_0=-\sum_{i=1}^{k+1} \vol_k(F_i) u_i,$$
we deduce 
$$ c_\Psi(\bar F_i)=c_\Psi(\bar F_0)$$
for $i\in\{1,\ldots,k+1\}$. In fact this shows that $c_\Psi$ attains the same value on all $k$-dimensional, linear subspaces contained in a fixed $(k+1)$-dimensional, linear subspace.
Since for two arbitrary $k$-dimensional,
linear subspaces $E_1$ and $E_2$ there exists a sequence of $k$-dimensional,
linear subspaces starting with $E_1$ and ending with $E_2$ such that two consecutive subspaces are contained in some $(k+1)$-dimensional, linear subspace,
 we conclude that $c_\Psi$ is constant and hence $\Psi=c_0\Delta_k$ for some number $c_0$.

Conversely, assume now that $\Psi$ is a linear combination of the $\Delta_k$. It is a  well-known fact, however, that for every convex body $K$ the measures $\Delta_k(K)$ have their centroids
 at the origin, see e.g.\ \cite{schneider_book}*{p.~281}. In other words, $C(\Delta_k)=0$ for $k=0,1,\ldots, n-1$ and therefore $C(\Psi)=0$. 
\end{proof}

An important class of examples of angular area measures is provided by constant coefficient area measures. To define constant coefficient area measures we need to introduce a certain variation of the normal cycle of a convex body $K$ which comes from replacing the sphere bundle by the disc bundle $V\times B(V)$ in the definition of $N(K)$. A similar construction was introduced by Bernig and Fu in \cite{bernig_fu06}, the difference to our definition here is that we remove the zero section from the disc bundle.
$$N_1(K)=\{(x,u)\in K\times V: \ 0<|u|\leq 1\ \text{and}\ u\ \text{is a normal of}\ K\ \text{at}\ x\}.$$
Observe that $N_1(K)$ has a boundary, namely $\partial N_1(K)=N(K)$. 
\begin{definition}
We call an area measure $\Psi \in\Area(V) $ a \emph{constant coefficient area measure} if there exists a constant coefficient form  $\omega\in \Lambda^n(V^*\times V^*)\subset\Omega^n(V\times V)$ such that
$$\Psi(K,A)=\int_{N_1(K)\cap p^{-1}(A)} \omega,$$
whenever $K\in \calK(V)$ and $A\subset S(V)$ is a Borel set. Here $p:V\times (B(V)\setminus\{0\}) \rightarrow S(V)$ is given by $p(x,v)=v/|v|$. 
\end{definition}

We denote by $N_1(F,P)$ the set of normals $v$ of $P$ at $F$ which satisfy $0<|v|\leq 1$. 

\begin{remark} Constant coefficient valuations were introduced by Bernig and Fu in \cite{bernig_fu11} and constant coefficient curvature measures were introduced by Bernig, Fu, and Solanes in \cite{bernig_etal12}.
\end{remark}

\begin{lemma}
Every constant coefficient area measures is angular.
\end{lemma}
\begin{proof}
Let $\Psi$ be a constant coefficient area measure given by some $\omega\in\Lambda^n(V^*\times V^*)$. By linearity, it is sufficient to prove the lemma under the additional assumption that $\omega=\pi^*_1\omega_1\wedge\pi^*_2\omega_2$ with $\omega_1\in\Lambda^kV^*$ and $\omega_2\in\Lambda^{n-k}V^*$. 

Let $f$ be a smooth function on the unit sphere and let $P$ be a polytope. Then clearly
\begin{align*}
[N_1(P)](p^*f\wedge \omega)	&= \sum_{F\in\calF_k(P)} [F\times N_1(F,P)](p^*f\wedge \omega)\\
							&= \sum_{F\in\calF_k(P)}\; [F](\omega_1) [N_1(F,P)](p^*f\wedge \omega_2),
\end{align*}
where $[M]$ denotes the current which is given by integration over the manifold $M$. Since both $\omega_1$ and $\omega_2$ have constant coefficients, we obtain
$$[F](\omega_1)\; [N_1(F,P)](p^*f\wedge \omega_2)= c_\Psi(\bar F)\; \frac{\int_{S(V)\cap N(P,F)} f\; d\calH^{n-1-k}}{(n-k)\omega_{n-k}} \;\vol_k(F),$$
with some constant $c_\Psi(\bar F)$ depending only on the $k$-dimensional, linear subspace parallel to the face $F$. This proves the lemma.
\end{proof}

When does $\beta\in\Omega^{n-1}(SV)$ define a constant coefficient area measure? The following proposition gives a sufficient condition. 
For $(x,v)\in V\times (V\setminus \{0\})$ we put $r(x,v)=|v|$ and consider the radial vector field 
$$R=\grad r.$$

\begin{proposition}\label{prop_constantcoeff}
Suppose $\beta\in\Omega^{n-1}(V\times V)$ is translation-invariant in the first factor and $d\beta$ has constant coefficients. Then
$$\Psi(K,A)=\int_{N(K)\cap\pi_2^{-1}(A)} \beta $$
has constant coefficients if (i) $\beta$ is homogeneous of degree $0\leq k<n-1$ (in the variables of the first factor) and $i_R\beta=0$; or (ii)  $\beta$ is homogeneous of degree $ n-1$ and the
 coefficients of $\beta$ are linear functions. In both cases $\Psi$ is in particular angular.
\end{proposition}
\begin{proof}

For $K\in \calK^{sm}$ we define a diffeomorphism $\exp: \partial K\times (0,1]\rightarrow N_1(K)$ by
 $$\exp(x,t)=\exp_t(x)=(x,t\nu(x)).$$
Here $\nu(x)$ denotes the outer unit normal of $K$ at $x$. If $\omega\in \Omega^n(V\times V)$, then clearly
\begin{equation}\label{eq_pullbackexp}
\int_{N_1(K)} \omega = \int_{\partial K\times (0,1]} \exp^*\omega = \int_0^1 \left( \int_{\partial K} \exp_t^*(i_R\omega)\right) dt.
\end{equation}
Fix now a smooth function $f$ on the unit sphere and put $\tilde f:	=p^* f$ for its $0$-homogeneous extension to the disc bundle with zero section removed. Furthermore let $\eta_\varepsilon$ be the smooth cut-off function $\eta_\varepsilon(x,v)=h(|v|/\varepsilon)$, $0<\varepsilon<1$, where
$$h(t)=\left\{ \begin{array}{ll} 	1-e^{1-\frac{1}{1-t ^2}}	& \qquad t\in(0,1)\\
									0							& \qquad t\leq 0\\
									1							& \qquad t\geq 1

\end{array} \right.$$
Since the smooth form $\eta_\epsilon \pi_2^*f \wedge \beta$ is compactly supported on $N_1(K)$ and $\partial N_1(K)=N(K)$, we can use Stokes' theorem to obtain
$$\int_{N(K)}  \pi_2^*f \wedge \beta=  \int_{N_1(K)} d(\eta_\varepsilon \tilde f)\wedge \beta + \int_{N_1(K)}\eta_\varepsilon\; \tilde f \wedge d\beta.$$
To prove that $\beta $ defines a constant coefficient area measure it is therefore sufficient to show that 

\begin{equation}\label{eq_cutofflimit}
\lim_{\varepsilon \rightarrow 0}  \int_{N_1(K)} d(\eta_\varepsilon \tilde f)\wedge \beta =0.
\end{equation}
Using \eqref{eq_pullbackexp}, $i_R\beta=i_R d\tilde{f}=0$ and $i_R d\eta_\varepsilon(x,v)=\frac{1}{\varepsilon}h'(|v|/\varepsilon)$, we obtain

\begin{align*}
\int_{N_1(K)} d(\eta_\varepsilon \tilde f)\wedge \beta &=\int_0^1 \left( \int_{\partial K} \exp_t^*(\tilde f\wedge i_Rd\eta_\varepsilon\wedge \beta)\right) dt\\
&=\frac{1}{\varepsilon}\int_0^\varepsilon h'(t/\varepsilon) \left( \int_{\partial K} \exp_t^*(\tilde f\wedge\beta)\right) dt\\
&=\int_0^1 h'(t) \left( \int_{\partial K} \exp_{\varepsilon t}^*(\tilde f\wedge\beta)\right) dt\\
&=\int_0^1 h'(t) \left( \int_{N(K)} \pi_2^*f\wedge m_{\varepsilon t}^*\beta\right) dt,\\
\end{align*}
where the last line follows from $\exp_1(\partial K)=N(K)$ and $m_{\lambda}$ denotes multiplication in the second component, $m_{\lambda}(x,v)=(x,\lambda v)$. 
If $\beta$ is homogeneous of degree $0\leq k<n-1$, then there clearly exists a constant $C$, depending only on $\beta$, such that 
$$\| m_{\varepsilon t}^* \beta\| \leq C \varepsilon \qquad \text{on}\ V\times B(V)$$
whenever $0\leq \varepsilon, t\leq 1$. Here $\|\cdot\|$ denotes the comass norm, see \cite{federer69}*{1.8.1}. 
If $\beta$ is homogeneous of degree $k=n-1$, then assumption (ii) assures that the above bound holds as well. Therefore
$$\left| \int_0^1 h'(t) \left( \int_{N(K)} \pi_2^*f\wedge m_{\varepsilon t}^*\beta\right) dt\right| \leq C' \varepsilon,$$
for some constant $C'$. This proves \eqref{eq_cutofflimit} and the proposition.

\end{proof}

\section{Unitarily invariant area measures and their centroids} \label{sec_unitarilyinv}

In the previous section we have introduced smooth area measures and established some general properties of smooth area measures. 
In particular, we have shown that the space of smooth area measures is a non-trivial module over smooth valuations. 
In the present section we introduce the class of unitarily invariant, 
smooth area measures as a hermitian analogue of the classical area measures of convex bodies. 
The results obtained for general smooth area measures can be strengthened and made explicit for unitarily invariant area measures.
 The main results of this section are Theorem~\ref{prop_dimker}, which shows that the subspace of unitarily invariant area measures
 which arise as the first variation of unitarily invariant valuations coincides precisely with the kernel of the centroid map, and Theorem~\ref{thm_dim}, 
which gives the dimension of the vector space of unitarily equivariant valuations. As an application of Theorem~\ref{thm_dim}, 
we obtain a new characterization of the Steiner point map in hermitian vector spaces. 
The explicit description of the module of unitarily invariant area measures is given in Section~\ref{sec_module}.

\subsection{Unitarily invariant area measures}

In this section and in the rest of the article we assume that the underlying vector space $V$ equals $\CC^n$.  The standard action of the unitary group $U(n)$ on $\CC^n$ induces a natural action on $S\CC^n=\CC^n\times S^{2n-1}$. Explicitly, the action is given by the restriction of the diagonal action of $U(n)$ on $T\CC^n\cong \CC^n\times\CC^n$. Furthermore, we denote by $\overline{U(n)}= U(n)\ltimes \CC^n$ the group of unitary affine transformations of $\CC^n$ and we let $\overline{U(n)}$ act in the obvious way on $T\CC^n$ and $S\CC^n$. 

\begin{definition}
We call a smooth area measure $\Psi$  \emph{unitarily invariant} or \emph{$U(n)$-invariant} if 
$$\Psi(g K, g A)=\Psi(K,A)$$
whenever $g\in U(n)$, $K\in\calK(\CC^n)$, and $A\subset S^{2n-1}$ is a Borel set.
The space of all unitarily invariant area measures is denoted by $\Area^{U(n)}$. 
\end{definition}

We denote by $\Val^{U(n)} \subset \Val^{sm}$ the subspace of unitarily invariant valuations.

\begin{remark}
\begin{itemize}
\item[(1)] If $\Psi$ is an unitarily invariant area measure, then there exists a $\overline{U(n)}$-invariant, smooth $(n-1)$-form $\omega$ on the sphere bundle such that
$$\Psi(K,A)=\int_{N(K)\cap\pi^{-1}_2(A)}\omega$$
whenever $K\in\calK(\CC^n)$ and $A\subset S^{2n-1}$ is Borel. 
Indeed, since $\Psi$ is a smooth area measure, it is represented by some translation-invariant differential form on the sphere bundle. Since the unitary group is compact, we can average with respect to the Haar probability measure to obtain a $U(n)$-invariant differential form. 

\item[(2)] If $\Psi\in \Area^{U(n)}$, then clearly $\glob(\Psi)\in\Val^{U(n)}$. 
In fact, every unitarily invariant valuation in $\bigoplus_{k=0}^{2n-1}\Val_k^{U(n)}$ is the globalization of some unitarily invariant area measure, see \eqref{eq_globarea} below.
\item[(3)] If $\mu\in\Val^{U(n)}$, then $\delta\mu\in\Area^{U(n)}$. 

\end{itemize}

\end{remark}

Since every unitarily invariant area measure is represented by an  $\overline{U(n)}$-invariant form on $S\CC^n$, we start our investigation of unitarily invariant area measures with an explicit description of the algebra of $\overline{U(n)}$-invariant forms on $S\CC^n$. To this end we denote by $(z_1,\ldots,z_n,\zeta_1,\ldots,\zeta_n)$ the canonical coordinates on $\CC^n\times\CC^n$, $z_i=x_i+\ii y_i$ and $\zeta_i=\xi_i+\ii \eta_i$. As in \cite{bernig_fu11} we consider the $\overline{U(n)}$-invariant $1$-forms
\begin{align*}
	\alpha & = \sum_{i=1}^n \xi_idx_i +\eta_i dy_i,\\
	\beta & = \sum_{i=1}^n \xi_idy_i -\eta_i dx_i,\\
	\gamma & = \sum_{i=1}^n \xi_id\eta_i -\eta_i d\xi_i,
\end{align*}
and the $\overline{U(n)}$-invariant $2$-forms 
\begin{align*}
    \theta_0&=\sum_{i=1}^n d\xi_i\wedge d\eta_i,\\
    \theta_1&= \sum_{i=1}^n dx_i\wedge d\eta_i- dy_i\wedge d\xi_i,\\
    \theta_2&= \sum_{i=1}^n dx_i\wedge dy_i,\\
    \theta_s&=\sum_{i=1}^n dx_i\wedge d\xi_i+ dy_i\wedge d\eta_i
\end{align*}
on $T\CC^n$. The restrictions of these forms to the sphere bundle generate the algebra of $\overline{U(n)}$-invariant forms on $S\CC^n$, see \cite{bernig_fu11} or, for the case of $\overline{SU(n)}$-invariant forms, \cite{bernig09}. Observe that $\alpha$ is precisely the canonical contact form on $S\CC^n$, $d\alpha=-\theta_s$, and that the Reeb vector field $T$ on $S\CC^n$ is given in coordinates  by
$$T=\sum_{i=1}^n \xi_i\frac{\partial}{\partial x_i} + \eta_i\frac{\partial}{\partial y_i}.$$

For non-negative integers $k,q$ with $\max\{0,k-n\}\leq q\leq \frac{k}{2}<n$ we put as in \cites{bernig_fu11,bernig_etal12}
\begin{align*}
\beta_{k,q}&=c_{n,k,q}\beta\wedge\theta_0^{n-k+q}\wedge\theta_1^{k-2q-1}\wedge\theta_2^q,\qquad q<\frac{k}{2},\\
\gamma_{k,q}&=\frac{c_{n,k,q}}{2}\gamma\wedge\theta_0^{n-k+q-1}\wedge\theta_1^{k-2q}\wedge\theta_2^q, \qquad k-n< q,
\end{align*}
where 
$$c_{n,k,q}=\frac{1}{q!(n-k+q)!(k-2q)!\omega_{2n-k}}.$$
We denote by $B_{k,q}$ and $\Gamma_{k,q}$ the area measures represented by $\beta_{k,q}$ and $\gamma_{k,q}$, respectively. Since the normal cycle vanishes on forms which are multiplies of $\alpha$ or $d\alpha$, we see that $\Area^{U(n)}$ is spanned by $B_{k,q}$ and $\Gamma_{k,q}$. We will see below that these area measures form in fact a basis of $\Area^{U(n)}$. We know from \cite{bernig_fu11}*{Proposition 3.4} that
\begin{equation}\label{eq_globarea}\glob(B_{k,q})=\glob(\Gamma_{k,q})=\mu_{k,q}.\end{equation}                                                                
Here the $\mu_{k,q}$ denote the \emph{hermitian intrinsic volumes}, see \cite{bernig_fu11}.

\begin{definition}
We define $\Delta_{k,q}\in\Area^{U(n)}$ by
$$\Delta_{k,q}=\frac{k-2q}{2n-k}B_{k,q} + \frac{2(n-k+q)}{2n-k}\Gamma_{k,q},\qquad \max\{0,k-n\}\leq q\leq \frac{k}{2}<n.$$
In particular, $\Delta_{2q,q}=\Gamma_{2q,q}$ and $\Delta_{k,k-n}=B_{k,k-n}$.
For $k>2q$, $q>k-n$ we  also define
\begin{align*}
N_{k,q}	&=\Delta_{k,q}-B_{k,q}\\
		&=\frac{2(n-k+q)}{2n-k}(\Gamma_{k,q}-B_{k,q}).
\end{align*}
\end{definition}

\begin{remark}
These definitions mimic the definitions for unitarily invariant curvature measures which were first introduced by Bernig, Fu, and Solanes in \cite{bernig_etal12}. The important point to note is that  
$$\glob(\Delta_{k,q})=\mu_{k,q}\qquad \text{and}\qquad \glob(N_{k,q})=0.$$
So far we do not know whether always $N_{k,q}\neq 0$. This will follow from Proposition~\ref{prop_basis}. 
\end{remark}

\begin{lemma}\label{lem_deltaangular}
The subspace of angular area measures is spanned by the $\Delta_{k,q}$.

\end{lemma}
\begin{proof}
	Put $r^2=\sum_{i=0}^n \left(\xi_i^2+\eta_i^2\right)$ and let $R=\grad r$ be the radial vector field on $\CC^n\times (\CC^n\setminus\{0\})$. Since $i_R\beta=i_R\gamma=i_R\theta_2=0,$
	$$i_R\theta_0=r^{-1}\gamma,\qquad\text{and}\qquad  i_R\theta_1=-r^{-1}\beta,$$
	we obtain
	$$i_R\left(\frac{k-2q}{2n-k}\beta_{k,q} + \frac{2(n-k+q)}{2n-k}\gamma_{k,q}\right)=0.$$
	Proposition~\ref{prop_constantcoeff} implies that the $\Delta_{k,q}$ are angular. For dimensional reasons they span the subspace of angular area measures.
\end{proof}

\begin{proposition}\label{prop_basis}
The area measures $\{B_{k,q}\}\cup\{\Gamma_{k,q}\}$ form a basis of $\Area^{U(n)}$. The same holds true for  $\{\Delta_{k,q}\}\cup\{N_{k,q}\}$.
\end{proposition}
\begin{proof}
Since the area measures $\{B_{k,q}\}\cup\{\Gamma_{k,q}\}$ span $\Area^{U(n)}$, we have  
$$\dim\Area^{U(n)}\leq n^2+n+1= 2\dim\Val^{U(n)}-2n-1.$$
To prove the proposition, it is sufficient to show that the above inequality is in fact an equality. 
Consider the restriction of the centroid map to the space of unitarily invariant area measures. We denote it by the same symbol $C$. Since $C$ is a linear map, we clearly have
\begin{equation}\label{eq_kerCimC}
\dim \Area^{U(n)}=\dim \ker C + \dim \img C
\end{equation}
Since the kernel of the first variation map is $1$-dimensional by Proposition~\ref{thm_firstvar} and since $C\circ \delta=0$ by Lemma~\ref{lem_centroids}, we obtain
$$
\dim \ker C\geq \dim \img (\delta|_{\Val^{U(n)}})= \dim \Val^{U(n)}-1.
$$
Furthermore, the angularity of the $\Delta_{k,q}$ (see Lemma~\ref{lem_deltaangular}) and Theorem~\ref{thm_angularity} yield
\begin{equation}\label{eq_dimimC}
\dim \img C\geq  \dim\Val^{U(n)}-2n.
\end{equation}
Thus,
\begin{equation}\label{eq_dimarea}
\dim\Area^{U(n)}=2\dim\Val^{U(n)}-2n-1.
\end{equation}
\end{proof}

\begin{corollary} \label{cor_dimarea}
$\dim\Area_k^{U(n)}=\dim\Val_{k}^{U(n)}+\dim\Val_{k+1}^{U(n)}-1$, $0\leq k\leq 2n-1$.
\end{corollary}

The next proposition characterizes those unitarily invariant area measures which lie in the kernel of the centroid map. 
\begin{theorem}\label{prop_dimker}
Let $\Psi\in \Area^{U(n)}$. Then 
$$C(\Psi)=0\quad \text{if and only if}\quad \Psi=\delta\phi $$
for some $\phi\in\Val^{U(n)}$.

\end{theorem}
\begin{proof}
Since $C\circ\delta =0$, it is sufficient show that the kernel of the map $C:\Area^{U(n)}\rightarrow \Vector^{sm}$ has dimension at most $\dim \Val^{U(n)}-1$. This follows immediately from \eqref{eq_kerCimC}, \eqref{eq_dimimC}, and \eqref{eq_dimarea}.
\end{proof}

\subsection{Unitarily equivariant valuations} 

If $\varphi\in \Vector(\CC^n)$ is the image of a unitarily invariant area measure under the centroid map, then clearly
$$ \varphi(gK)= g\varphi(K)$$
for every $g\in U(n)$ and $K\in\calK(\CC^n)$. This motivates the following definition.

\begin{definition} 
We call a $\CC^n$-valued valuation $\varphi\in \Vector(\CC^n)$ \emph{unitarily equivariant} or $U(n)$-\emph{equivariant} if 
$$ \varphi(gK)= g\varphi(K)$$
for every $g\in U(n)$ and $K\in\calK(\CC^n)$. The vector space of $U(n)$-equivariant valuations is denoted by $\Vector^{U(n)}$.
\end{definition}

As a vector space, the set of translation-invariant, continuous valuations with values in $V$ is naturally isomorphic to $\Val(V)\otimes V$. Under this isomorphism, the natural $GL(V)$-action on $V$-valued valuations,
$$(g\cdot \varphi)(K)= g\varphi(g^{-1}K),$$
corresponds to the standard action of $GL(V)$ on the tensor product $\Val\otimes V$. There exists no non-trivial, continuous, translation-invariant,
 $V$-valued valuation which is $SO(V)$-equivariant, see e.g.\ \cite{alesker_etal}. We can recast this in representation theoretic terms by saying that the subspace of $SO(V)$-invariant elements of
 $\Val\otimes V$ is trivial, i.e.\ we have $(\Val\otimes V)^{SO(V)}=\{0\}$.

Recall that the irreducible complex representations of $SO(2n)$ are parametrized by their highest weights which are tuples of integers $\lambda=(\lambda_1,\ldots,\lambda_n)$
satisfying $\lambda_1\geq \ldots \geq \lambda_{n-1}\geq |\lambda_n|$, see e.g.\ \cite{broecker_dieck85}*{p. 274}. It is possible to describe the decomposition of tensor products
of irreducible representations in terms of this parametrization.  For example an application of Klimyk's formula (see e.g.\ \cite{humphreys72}*{Ex. 24.9})
yields
$$\Gamma_\lambda \otimes \CC^{2n}=\sum_\nu \Gamma_\nu,$$
where the sum extends over all  $\nu$ satisfying $\nu= \lambda \pm e_i$ for some $i$. 

The next lemma follows from Helgason's theorem (see e.g. \cite{takeuchi94}*{p. 151}) applied to the symmetric space $SO(2n)/U(n)$;  
it has been used in \cite{alesker03} to compute the dimension of the space of unitarily invariant valuations. 

\begin{lemma}
	In every irreducible complex $SO(2n)$-representation the subspace of $U(n)$-invariant vectors is at most $1$-dimensional. 
This subspace is $1$-dimensional if and only if the highest weight $\nu$ of the irreducible  $SO(2n)$-representation satisfies 

\begin{itemize}
	\item[(i)]
\begin{equation}\label{eq_un_even}\nu_1=\nu_2\geq\nu_3=\nu_4\geq \cdots\geq\nu_{n-1}=\nu_n\end{equation}
if $n$ is even; or
\item[(ii)] 
\begin{equation}\label{eq_un_odd}\nu_1=\nu_2\geq\nu_3=\nu_4\geq \cdots\geq\nu_{n-2}=\nu_{n-1}\geq \nu_n=0\end{equation}
if $n$ is odd.
\end{itemize}
\end{lemma}

We denote by $\Vector_k\subset \Vector$ the subspace of $k$-homogeneous valuations. Observe that both $\Vector(\CC^n)$ and $\Vector_k(\CC^n)$ carry the structure of a complex vector space. 

\begin{theorem}\label{thm_dim}
 $$\dim_\CC \Vector^{U(n)}_k=\dim_\RR \Val_k^{U(n)}-1.$$

\end{theorem}
\begin{proof}
If $W$ is a real vector space, we denote by $W_\CC$ its complexification. The decomposition of  $\Val_{k,\CC}$ under the action of $SO(2n)$ into isotypical components was determined in \cite{alesker_etal}. 
It was shown that the representation of $SO(2n)$ on  $\Val_{k,\CC}$ is multiplicity-free and that $\Val_{k,\CC}=\bigoplus_{\lambda} \Gamma_\lambda$,
where $\lambda$ satisfies 
\begin{equation}\label{eq_weights}|\lambda_i|\neq 1\quad \forall i,\quad |\lambda_2|\leq 2, \quad \text{and}\quad  \lambda_i=0\quad \text{for}\quad i>\min\{k,2n-k\}.\end{equation}
Put $V=\CC^n$. Since clearly
$$\Vector_k(V)_\CC\cong \Val_{k,\CC}\otimes V_\CC =  \bigoplus_{\lambda} \Gamma_\lambda\otimes V_\CC,$$
 we obtain
$$\dim_\CC (\Vector_{k,\CC})^{U(n)}=\sum_\lambda\dim_\CC(\Gamma_\lambda\otimes V_\CC)^{U(n)}.$$
We claim that $\dim_\CC(\Gamma_\lambda\otimes V_\CC)^{U(n)}=2$ if $\lambda$ satisfies 
$$\lambda_1=3, \quad \lambda_2=\cdots=\lambda_{2m}=2, \quad \text{and}\quad \lambda_i=0 \quad \text{for}\quad i>2m$$
for some integer $1\leq m\leq \min\left\{\left\lfloor \frac{k}{2}\right\rfloor,\left\lfloor \frac{2n-k}{2}\right\rfloor\right\}$ and that $\dim_\CC(\Gamma_\lambda\otimes V_\CC)^{U(n)}=0$ otherwise.
In fact, fix some $\lambda$ satisfying \eqref{eq_weights} and suppose that $\nu = \lambda + e_j$ for some $j$.  
If we require $\nu$ to satisfy either (\ref{eq_un_even}) or (\ref{eq_un_odd}), then necessarily $\nu_1=\nu_2=3$ and $\lambda_1=3$, $\lambda_2=\cdots=\lambda_{2m}=2$, and $\lambda_i=0$ for $i>2m$. 
If $\nu=\lambda-e_j$, then either (\ref{eq_un_even}) or (\ref{eq_un_odd}) force $\nu_1=\nu_2=2$ and $\lambda_1=3$, $\lambda_2=\cdots=\lambda_{2m}=2$, and $\lambda_i=0$ for $i>2m$. Now  $\dim_\RR \Vector_k^{U(n)}= \dim_\CC (\Vector_{k,\CC})^{U(n)}$ and hence
$$\dim_\RR \Vector_k^{U(n)}= 2 \min\left\{\left\lfloor \frac{k}{2}\right\rfloor,\left\lfloor \frac{2n-k}{2}\right\rfloor\right\}=2(\dim_\RR \Val_k^{U(n)}-1),$$
where the second identity follows from \cite{alesker01}. Since $2 \dim_\CC \Vector_k^{U(n)}= \dim_\RR \Vector_k^{U(n)}$, we obtain $\dim_\CC \Vector^{U(n)}_k=\dim_\RR \Val_k^{U(n)}-1$. 
\end{proof}

\begin{corollary}
$\Vector^{U(n)}\subset \Vector^{sm}$.
\end{corollary}
\begin{proof}
Using the notation of the proof of Theorem~\ref{thm_dim}, we have 
$$\bigoplus_\lambda \Gamma_\lambda \otimes V_\CC \subset \Val_\CC^{sm}\otimes V_\CC = (\Val_\CC\otimes  V_\CC)^{sm} =\Vector_\CC^{sm},$$
where the sum extends over all $\lambda$ satisfying \eqref{eq_weights}.
The corollary follows at once from the fact that $\Vector^{U(n)}$ is finite-dimensional.
\end{proof}

As an application of the above theorem let us give a new characterization of the Steiner point map in hermitian vector spaces. Recall that the Steiner point of $K\in\calK(V)$ is given by
$$s(K)=\frac{1}{n} \int_{S(V)} uh_K(u)\; du,$$
where $h_K$ is the support function of $K$ and $du$ denotes integration with respect to the rotation-invariant probability measure on the unit sphere. For more information on the Steiner point see \cites{schneider71,schneider_book} and the references there. By a theorem of Schneider \cite{schneider71}, the Steiner point map $s:\calK(V)\rightarrow V$ is the unique continuous map with the properties that
\begin{itemize}
	\item[(i)] $s(K+L)=s(K)+s(L)$ for $K,L\in\calK(V)$; and
	\item[(ii)] $s\circ g=g\circ s$ for $g\in \overline{SO(V)}$.
\end{itemize} 
 
As a consequence of Theorem \ref{thm_dim}, we obtain that the Steiner point map is already characterized by $\overline{U(n)}$-equivariance.
\begin{theorem}
	Let $f:\calK(\CC^{n})\rightarrow \CC^n$ be a continuous map which satisfies
\begin{itemize}
	\item[(i)] $f(K+L)=f(K)+f(L)$ whenever $K,L\in\calK(\CC^n)$; and
	\item[(ii)] $f\circ g=g\circ f$ whenever $g\in \overline{U(n)}$.
\end{itemize}
Then $f=s$.
\end{theorem}
\begin{proof}
Let $K,L\in\calK(\CC^n)$. Since $K\cup L + K\cap L= K+L$
whenever $K\cup L$ is convex, we see that $f$ is a valuation.
Using the continuity of $f$ and (i) it is not difficult to see that $f$ is $1$-homogeneous.  Since the Steiner point map is in particular $\overline{U(n)}$-equivariant, we conclude that $f-s$ is unitarily equivariant and translation-invariant. Thus, $f-s\in \Vector_1^{U(n)}=\{0\}$. 
\end{proof}

\section{The module of unitarily invariant area measures}
\label{sec_module}

With applications to hermitian integral geometry in mind, the goal of this section is to determine the action of unitarily invariant valuations on unitarily invariant area measures as explicitly as possible. The basis of these investigations is the explicit description of the algebra of unitarly invariant valuation by Bernig and Fu \cites{fu06, bernig_fu11}. We begin by recalling their results. 

\subsection{The unitary valuation algebra}

Following \cites{fu06,bernig_fu11}, we consider the unitarily invariant valuations 
\begin{equation}\label{eq_tsvaluations}
t=\frac{2}{\pi}\mu_{1,0}\quad\text{and}\quad s=\frac{1}{\pi}\left( \mu_{2,1}+ \frac{1}{2}\mu_{2,0}\right).
\end{equation}
These valuations have various special properties, see \cite{fu06}. In particular, if we equip the finite-dimensional vector space $\Val^{U(n)}$ with the Alesker product, then $s$ and $t$ generate this algebra. We denote by $\RR[s,t]$ the polynomial algebra in two variables $s$ and $t$. 

\begin{theorem}[Fu \cite{fu06}]\label{thm_fu} The algebra $\Val^{U(n)}$ is generated by two elements. More precisely,
$$\Val^{U(n)}\cong \RR[s,t]/(f_{n+1},f_{n+2}),$$
where the polynomials $f_k$ are determined by the Taylor series expansion
$$\log(1+tx+sx^2)=\sum_{k=0}^\infty f_k(s,t) x^k.$$
\end{theorem}
\begin{remark}
Instead of using the Alesker product, we could also equip the vector space of unitarily invariant valuations with the convolution product of Bernig and Fu. By the properties of the Fourier transform \eqref{eq_fourierhom} however, the algebras $(\Val^{U(n)},\;\cdot\;)$ and $(\Val^{U(n)},\;*\;)$ are isomorphic. 

\end{remark}
Explicitly, the Fu polynomial $f_k=f_k(s,t)$ is given by 
$$f_k=(-1)^{k+1}\sum_{q=0}^{\lfloor k/2\rfloor}  \frac{(-1)^q}{k-2q} \binom{k-q-1}{q} s^qt^{k-2q}.$$
Following \cite{bernig_fu11}, we put $u=4s-t^2$. In terms of the basis given by the hermitian intrinsic volumes,
\begin{equation}\label{eq_uvaluation}
u=\frac{2}{\pi}\mu_{2,0}.
\end{equation}
By \cite{bernig_fu11}*{Proposition 3.5}, the Fu polynomial can be expressed in terms of $t$ and $u$ as
\begin{equation}\label{eq_futu}
f_{k}=\frac{1}{k(-2)^{k-1}}\sum_{q=0}^{\lfloor k/2\rfloor} (-1)^q\binom{k}{2q} t^{k-2q} u^q. 
\end{equation}

If we consider $\CC^n$ as a subset of $\CC^{n+1}$ in the natural way, then the sequence of inclusions
$$\CC^1\subset\CC^2\subset \CC^3\subset\ldots$$
induces a sequence of restrictions
\begin{equation}\label{eq_inverselimit}
\Val^{U(1)}\leftarrow\Val^{U(2)}\leftarrow\Val^{U(3)}\leftarrow\ldots
\end{equation}
By the properties of the Alesker product, each restriction map is a homomorphism of algebras. The inverse limit of the system \eqref{eq_inverselimit} is denoted by $\Val^{U(\infty)}:=\varprojlim\Val^{U(n)}$ and called the algebra of global valuations, see \cite{bernig_fu11}. It was shown by Fu \cite{fu06} that $\Val^{U(\infty)}\cong \RR[s,t]$. We say that two global valuations are \emph{equal locally at $n$} if their projections to $\Val^{U(n)}$ are equal.

\subsection{The main theorem}

With the module structure from Section~\ref{sec_valandarea} the vector space of unitarily invariant area measures becomes a module over unitarily invariant valuations.
In the following we consider $\Val^{U(n)} \oplus \Val^{U(n)}$ as a $(\Val^{U(n)}, \;\cdot\;)$-module under the diagonal action. 
The main result of this section is 

\begin{theorem}\label{thm_main}
The module of unitarily invariant area measures is generated by two elements. More precisely,
$$\Area^{U(n)}\cong (\Val^{U(n)}\oplus \Val^{U(n)})/I_n,$$
where $I_n$ is the submodule generated by the following pairs of valuations
$$(p_n,-q_{n-1})\quad \text{and}\quad (0,p_n),$$  
which are determined by the Taylor series expansions
$$\frac{1}{1+tx+sx^2}=\sum_{k=0}^\infty p_k(s,t) x^k$$ 
and 
$$ -\frac{1}{(1+tx+sx^2)^2}=\sum_{k=0}^\infty q_k(s,t) x^k.$$

\end{theorem}
The bigger part of this section is devoted to the proof of the above theorem.

It is not difficult to give explicit expressions for the polynomials $p_k$ and $q_k$,  
$$p_k=(-1)^k\sum_{q=0}^{\lfloor k/2\rfloor}  (-1)^q \binom{k-q}{q}s^q t^{k-2q}$$
and 
$$q_k=(-1)^{k+1}\sum_{q=0}^{\lfloor k/2\rfloor}  (-1)^q(q+1) \binom{k+1-q}{q+1}s^q t^{k-2q}.$$

In the following it will be useful to express the polynomials $p_k$ and $q_k$ also in $t$ and $u=4s-t^2$.
\begin{lemma} The polynomials $p_k$ and $q_k$ can be written as
 $$p_k=\frac{(-1)^k}{2^k}\sum_{q=0}^{\lfloor k/2\rfloor}  (-1)^q \binom{k+1}{2q+1}t^{k-2q}u^q$$
and 
 $$q_k=\frac{(-1)^{k+1}}{2^k}\sum_{q=0}^{\lfloor k/2\rfloor}  (-1)^q (q+1)\binom{k+3}{2q+3}t^{k-2q}u^q.$$
\end{lemma}

\begin{proof} In terms of the generating functions of $p_k$, $q_k$, and $f_k$, we have
\begin{align*}
 \frac{x^2}{4} \left(1+tx+x^2 \frac{u+t^2}{4}\right)^{-1} &= \frac{\partial }{\partial u} \log\left( 1+tx+x^2 \frac{u+t^2}{4}\right),\\
  -\frac{x^2}{4} \left(1+tx+x^2 \frac{u+t^2}{4}\right)^{-2} &= \frac{\partial }{\partial u} \left( 1+tx+x^2 \frac{u+t^2}{4}\right)^{-1},
\end{align*}
and hence
$$p_k = 4 \frac{\partial }{\partial u} f_{k+2} \quad \text{and}\quad q_k = 4 \frac{\partial }{\partial u} p_{k+2}.$$
The lemma follows now from \eqref{eq_futu}. 
\end{proof}

We note that 
\begin{equation}\label{eq_piqi}-(4s-t^2)q_{k-1} +  tp_k=(k+1)^2f_{k+1}\end{equation}
which follows immediately from
$$\frac{(4s-t^2) x^2 }{(1+tx+sx^2)^2} + \frac{tx}{1+tx+sx^2} = \left(x\frac{\partial }{\partial x} \right)^2 \log(1+tx+sx^2).$$
We conclude this subsection with two local properties of the $p_k$. 

\begin{proposition}\label{lem_rel}
	$up_n=0$ and $t^np_n\neq 0$ as elements of $(\Val^{U(n)},\;\cdot\;)$. 
\end{proposition}
\begin{proof}
The first assertion follows at once from the global relation
\begin{equation}\label{eq_upi}(4s-t^2)p_n=2(n+2)f_{n+2}+(n+1)tf_{n+1},\end{equation}
which in turn follows from
$$\frac{(4s-t^2)x^2}{1+tx+sx^2} + 2xt= (2x+tx^2) \frac{\partial }{\partial x} \log(1+tx+sx^2).$$

To prove the second assertion we evaluate $t^np_n(B(\CC^n))$, where $B(\CC^n)$ denotes the unit ball in $\CC^n$. 
It was shown in \cite{fu06} that
\begin{equation}\label{eq_poincare}s^it^{2n-2i}(B(\CC^n))=\binom{2n-2i}{n-i}.\end{equation}
Hence, by the combinatorial identity (\ref{eq_binomsum}) below
$$t^np_n(B(\CC^n))=(-1)^n\sum_{i=0}^{\lfloor n/2\rfloor}  (-1)^i \binom{n-i}{i}\binom{2n-2i}{n-i}=(-1)^n2^n$$
and therefore $t^np_n\neq 0$, as claimed.
\end{proof}

\begin{lemma}
\begin{equation}\label{eq_binomsum}
\sum_{i=0}^{\lfloor n/2\rfloor} (-1)^i\binom{n-i}{i}\binom{2n-2i}{n-i}=2^n.	
\end{equation}
\end{lemma}
\begin{proof}
	To prove this combinatorial identity we use the `Snake Oil' method, see \cite{wilf94}*{p. 118}. In terms of generating functions (\ref{eq_binomsum}) may be written as 
$$\sum_{0\leq n} \sum_{i\leq n/2}  (-1)^i\binom{n-i}{i}\binom{2n-2i}{n-i}x^n=\frac{1}{1-2x}.$$
Interchanging the order of summation and using the formula 
$$\sum_{m=0}^\infty \binom{2m}{m} x^m=\frac{1}{\sqrt{1-4x}},\qquad |x|<\frac{1}{4},$$
the left-hand side may be expressed as 
\begin{align*}
	\sum_{0\leq i} \sum_{i\leq n-i}  (-1)^i\binom{n-i}{i}\binom{2n-2i}{n-i}x^n&= \sum_{0\leq i} \sum_{i\leq m}  (-1)^i\binom{m}{i}\binom{2m}{m}x^{m+i}\\
&= \sum_{m=0}^\infty \binom{2m}{m}(1-x)^m x^{m}\\
&= \frac{1}{\sqrt{1-4x(1-x)}}.
\end{align*}
Since $1-4x(1-x)=(1-2x)^2$, the sum equals 
$$\frac{1}{\sqrt{1-4x(1-x)}}=\frac{1}{1-2x}, \qquad |x|<\frac{1}{4},$$
as claimed.
\end{proof}

\subsection{Convolution with \texorpdfstring{$\hat s$}{s} and \texorpdfstring{$\hat t$}{t}} By the properties of the Fourier transform, the algebra $(\Val^{U(n)},\;*\;)$ is generated by the elements $\hat s$ and $\hat t$. Hence the first step to prove Theorem~\ref{thm_main} is to determine how  $\hat s$ and $\hat t$ act on unitarily invariant area measures. If we can derive explicit formulas in this case, then---in principle---we know how an arbitrary unitarily invariant valuation acts on $\Area^{U(n)}$. 

The following formulas have to be understood as follows: If for a certain pair of indices $k$ and $q$ one of the area measures on the right hand side does not exist, then it has to be replaced by $0$.

\begin{proposition}\label{prop_gammasubmod}
The subspace of $\Area^{U(n)}$ spanned by the area measures $\Gamma_{k,q}$ is a submodule. In particular,

$$\hat t* \Gamma_{k,q}=\frac{\omega_{2n-k+1}}{\pi\omega_{2n-k}}\bigg((k-2q+1)\Gamma_{k-1,q-1}+2(n-k+q+1)\Gamma_{k-1,q}\bigg)$$
and 

\begin{align*}\hat s*\Gamma_{k,q}=\frac{(k-2q+2)(k-2q+1)}{2\pi(2n-k+2)} & \Gamma_{k-2,q-2} \\
+&\frac{2(n-k+q+1)(n-q+1)}{\pi(2n-k+2)} \Gamma_{k-2,q-1}.
\end{align*}
\end{proposition}
\begin{proof}
Fix a valuation $\phi\in\Val^{U(n)}$ and assume that it is represented by an $\overline{U(n)}$-invariant form $\eta\in\Omega^{2n-1}(S\CC^n)$ satisfying $D\eta=d\eta$. By equation \eqref{eq_moduledef}, the convolution $\phi*\Gamma_{k,q}$ is represented by the form
\begin{equation}\label{eq_gammamultiple}
*_1^{-1} ( *_1 \gamma_{k,q} \wedge *_1  d\eta).
\end{equation}
Observe that the operator $*_1$ maps the subspace of forms which are multiples of the $1$-form $\gamma$ onto itself. Hence \eqref{eq_gammamultiple} is not only $\overline{U(n)}$-invariant, but also a multiple of $\gamma$ and therefore a linear combination of certain $\gamma_{k',q'}$. We conclude that the subspace spanned by the $\Gamma_{k,q}$ is a submodule. Since the globalization map is a module homomorphism and injective when restricted to the subspace spanned by the $\Gamma_{k,q}$, the formulas for $\hat t * \Gamma_{k,q}$ and $\hat s *\Gamma_{k,q}$ follow immediately from the expressions for $\hat t * \mu_{k,q}$ and $\hat s *\mu_{k,q}$ which are given in \cite{bernig_fu11}*{Lemma 5.2} and \cite{bernig_fu11}*{Corollary 5.10}. 
\end{proof}

By \eqref{eq_tsvaluations} and since $\widehat{\mu_1}=\mu_{2n-1}$ by \eqref{eq_fintrinsic}, we have
$$\hat{t} *\Psi=\frac{2}{\pi}\mu_{2n-1}*\Psi, \qquad \Psi\in\Area.$$
Together with Proposition~\ref{prop_gammasubmod} the following proposition gives a complete description of the action of $\hat t$ on $\Area^{U(n)}$ in terms of the measures $B_{k,q}$ and $\Gamma_{k,q}$.

\begin{proposition}\label{prop_taction}

\begin{align*}
\hat t*   B_{k,q}=&\frac{\omega_{2n-k+1}}{\pi\omega_{2n-k}}\bigg( (k-2q+1)B_{k-1,q-1} \\
						&+\frac{2(n-k+q+1)(k-2q-1)}{k-2q}B_{k-1,q}+\frac{2(n-k+q+1)}{k-2q}\Gamma_{k-1,q}\bigg)
\end{align*}

\begin{align*}
\hat t* N_{k,q}=\frac{\omega_{2n-k+1}}{\pi\omega_{2n-k}} \frac{2n-k+1}{2n-k}& \bigg( (k-2q+1)N_{k-1,q-1} \\
&+ \frac{2(n-k+q)(k-2q-1)}{k-2q}N_{k-1,q} \bigg).
\end{align*}
\end{proposition}
\begin{proof}

By Lemma~\ref{lem_convintrinsic}, all we have to do is to compute the Lie derivative of $\beta_{k,q}$ with respect to the Reeb vector field $T$. 
An easy computation shows that 
	$$\calL_T\gamma=\calL_T\theta_0=0,\qquad \calL_T\beta=\gamma, \qquad \calL_T\theta_1=2\theta_0,\qquad \calL_T\theta_2=\theta_1,$$
and hence we obtain

\begin{align*}\calL_T\beta_{k,q}= &q\frac{c_{n,k,q}}{c_{n,k-1,q-1}} \beta_{k-1,q-1} \\
& +2(k-2q-1)\frac{c_{n,k,q}}{c_{n,k-1,q}}\beta_{k-1,q}+2\frac{c_{n,k,q}}{c_{n,k-1,q}}\gamma_{k-1,q}.
\end{align*}
The formula for $\hat t* B_{k,q}$ follows immediately.
\end{proof}

By \eqref{eq_tsvaluations} and \cite{bernig_fu11}*{Theorem 3.2}, the Fourier transform of $s$ equals
$$\hat s=  \frac{1}{\pi}\left( \mu_{2n-2,n-1}+\frac{1}{2}\mu_{2n-2,n-2}\right).$$
Hence the valuation $\hat s$ can be represented by the $(2n-1)$-form
$$\omega=\frac{1}{\pi^2(n-2)!}\left(\frac{1}{2(n-1)}\gamma\wedge \theta_2^{n-1}+ \frac{1}{4}\beta\wedge\theta_1\wedge\theta_2^{n-2}\right).$$
Next we compute the Rumin differential of $\omega$. We do not really need an explicit formula, what is important is that $D\omega$  is a multiple of $\beta$. To increase readability, we will sometimes drop the $\wedge$-notation in the following. All products of forms are understood to be wedge products.
\begin{lemma}\label{lem_rumin}
The Rumin differential of $\omega$ equals
$$\frac{1}{4\pi^2(n-3)!} \alpha\wedge\beta\wedge(\theta_1^2+(d\alpha)^2)\wedge\theta_2^{n-3}.$$
In particular, $D\omega$ is a multiple of $\beta$.
	
\end{lemma}
\begin{proof}
	To simplify the notation we put $\omega'=\pi^2(n-2)!\omega$. Remember that $D\omega'=d(\omega'+\alpha\wedge \xi)$, where $\alpha \wedge\xi $ is the unique $(2n-1)$-form such that $d(\omega'+\alpha\wedge \xi)$ is a multiple of $\alpha$. We claim that we can choose
$$\xi=\frac{1}{4}(d\alpha +2\beta\wedge\gamma)\wedge\theta_2^{n-2}.$$
To prove this we use that $d(\omega'+\alpha\wedge\xi)$ is a multiple of $\alpha$ if and only if
\begin{equation}\label{eq_ruxi}\alpha \wedge d\omega'+ \alpha \wedge d\alpha \wedge \xi=0.\end{equation}
Since all forms involved are $U(n)$-invariant, it suffices to do the calculation at the point $(0,e_1)\in S\CC^n$. At this point 
$d\xi_1=0$, $\alpha=dx_1$, $\beta=dy_1$, and $\gamma=d\eta_1$. Next, we compute 
$$\alpha\wedge d\alpha\wedge\xi = \frac{1}{4}\alpha(d\alpha +\beta\wedge\gamma)^2\wedge\theta_2^{n-2}= -\frac{(n-2)!}{2}\sum_{i=2}^n dx_1d\xi_id\eta_i\bigwedge_{j=2}^n dx_jdy_j.$$
Similarly, we obtain 
$$\frac{1}{4}\alpha\wedge\theta_1^2\wedge\theta_2^{n-2}=-\frac{(n-2)!}{2}\sum_{i=2}^n dx_1d\xi_id\eta_i\bigwedge_{j=2}^n dx_jdy_j$$
and
$$\frac{1}{n-1}\alpha\wedge\theta_0\wedge\theta_2^{n-1}=(n-2)!\sum_{i=2}^n dx_1d\xi_id\eta_i\bigwedge_{j=2}^n dx_jdy_j.$$
Since $d\omega'=\frac{1}{n-1}\theta_0\wedge\theta_2^{n-1}+\frac{1}{4}\theta_1^2\wedge\theta_2^{n-2}$, we conclude that (\ref{eq_ruxi}) holds.

In particular, (\ref{eq_ruxi}) implies that $d\omega'+d\alpha\wedge\xi=\alpha\wedge i_T(d\omega'+d\alpha\wedge\xi)$ and hence we obtain
\begin{align*}
	D\omega' & =d(\omega'+\alpha\wedge\xi)=d\omega'+d\alpha\wedge\xi -\alpha\wedge d\xi\\
	&= \alpha\wedge i_T(d\omega'+d\alpha\wedge\xi)-\alpha\wedge d\xi\\
	&= \frac{n-2}{4}\alpha\wedge\beta\wedge(\theta_1^2+(d\alpha)^2)\wedge\theta_2^{n-3}
\end{align*}
\end{proof}

Again, the following formulas have to be understood as follows: If for a certain pair of indices $k$ and $q$ one of the area measures on the right hand side does not exist, then it has to be replaced by $0$.

\begin{proposition}\label{prop_saction}
\begin{align*}
\hat s * B_{k,q}=\frac{(k-2q+2)(k-2q+1)}{2\pi (2n-k+2)} & B_{k-2,q-2}\\
&+\frac{2(n-k+q+1)(n-q+1)}{\pi (2n-k+2)} B_{k-2,q-1}
\end{align*}

\end{proposition}
\begin{proof}
	Our starting point is the formula
$$s\cdot \mu_{k,q}=\frac{(k-2q+2)(k-2q+1)}{2\pi(k+2)}\mu_{k+2,q}+ \frac{2(q+1)(k-q+1)}{\pi(k+2)}\mu_{k+2,q+1},$$
which follows from \cite{bernig_fu11}*{Corollary 5.10}. Using the fact that
$\hat s*\mu_{k,q}=\FF(s\cdot \mu_{2n-k,n-k+q})$ and $\FF(\mu_{k,q})=\mu_{2n-k,n-k+q}$ (see \cite{bernig_fu11}*{Theorem 3.2}), we obtain a similar formula for $\hat s*\mu_{k,q}$.  It follows from Lemma~\ref{lem_rumin} that the Rumin differential of the form representing $\hat s$ is a multiple of the $1$-form $\beta$. Using this and \eqref{eq_moduledef}, we obtain that $\hat s* B_{k,q}$ is a linear combination of certain $B_{k',q'}$. Since the globalization map is injective when restricted to the subspace spanned by the $B_{k,q}$, we obtain the formula for $\hat s* B_{k,q}$ immediately from the formula for $\hat s*\mu_{k,q}$.

\end{proof}

\begin{definition}
We define maps  $\frakbb\colon \Val^{U(n)}\rightarrow \Area^{U(n)}$ and $\frakgg\colon\Val^{U(n)}\rightarrow \Area^{U(n)}$ by
$$\frakbb(\phi)=\hat \phi* B_{2n-1,n-1}$$
and
$$\frakgg(\phi)=\hat \phi* \Gamma_{2n-2,n-1}.$$
\end{definition}
Observe that if we view the algebra $(\Val^{U(n)},\;\cdot\;)$ as a module over itself,  then $\frakbb$ and $\frakgg$ become $\Val^{U(n)}$-module homomorphisms.

\begin{lemma}\label{lem_magic} Let $i$ and $j$ be non-negative integers. Then
\begin{enumerate}

\item

\begin{align*}
\frakbb(u^i)	&= \frac{4^i i!}{\pi^i} \left(B_{2n-2i-1,n-i-1} -\frac{2i}{2i+1}	 \Gamma_{2n-2i-1,n-i-1} \right)\\
			 	&= \frac{4^ii!}{(2i+1)\pi^i} \left(\Delta_{2n-2i-1,n-i-1} -2(i+1)	 N_{2n-2i-1,n-i-1} \right)
\end{align*}

and 

$$\frakgg(u^i)= \frac{(2n+1)!}{ n!\pi^i}\Delta_{2(n-i-1),n-i-1}.$$

\item

\begin{align*}
\binom{2i+2j+1}{2j} \frakbb(t^{2j}u^i)\equiv &\frac{4^{i+j}(i+j)!}{\pi^{i+j}} \times\\
 &\sum_{k=0}^{\min(j,n-i-j-1)} (2k+1)\binom{i+j-k}{i} B_{2(n-i-j)-1,n-i-j-k-1}
\end{align*} 
modulo $\spn\{\Gamma_{k,q}\}$.

\end{enumerate}

\end{lemma}
\begin{proof}
(1) follows immediately from Propositions \ref{prop_gammasubmod}, \ref{prop_taction}, and \ref{prop_saction} by induction on $i$. This implies in particular that (2) holds true for $j=0$. 
Denote the right-hand side of (2) by $S(i,j)$. Since the subspace spanned by the $\Gamma_{k,q}$ is invariant under the module action, it is sufficient to show that
$$\frac{(2i+2j+2)(2i+2j+3)}{(2j+1)(2j+2)} \hat t*\hat t* S(i,j)=S(i,j+1)$$
to finish the proof of (2) by induction on $j$.
But this follows from a simple, albeit long and tedious, calculation using Proposition~\ref{prop_taction}.

\end{proof}

\subsection{Proof of Theorem~\ref{thm_main}}

To simplify the notation, we put $\frakbb_k=\frakbb(\Val_k^{U(n)})$ and $\frakgg_k=\frakgg(\Val_k^{U(n)})$.

\begin{lemma}\label{lem_bgdim}

\begin{equation}\label{eq_dimbb}
\dim \frakbb_k=\dim\Val_k^{U(n)} \qquad 0\leq k<2n.
\end{equation}

\begin{equation}\label{eq_dimgg}\dim  \frakgg_k=\left\{ 
\begin{array}{ll} 
\dim\Val_k^{U(n)}& \qquad  0\leq k<n\\
\dim\Val_{k}^{U(n)}-1& \qquad  n\leq k<2n-1\\
\end{array}
\right.
\end{equation}
Moreover, the image of $\frakgg\colon\Val^{U(n)}\rightarrow \Area^{U(n)}$ coincides with the span of the $\Gamma_{k,q}$.
\end{lemma}
\begin{proof}
Since $2B_{2n-1,n-1}=S_{2n-1}$, we have $2 \frakbb(\phi)=\delta(\widehat{\phi})$ and \eqref{eq_dimbb} follows. 

To prove \eqref{eq_dimgg} first observe that by Proposition~\ref{prop_gammasubmod} the subspace spanned by the area measures $\Gamma_{k,q}$ is invariant under the module action.  Observe also that the globalization map restricted to this subspace is injective and the image of $\frakgg$ is contained in it. Thus,
$$\dim \frakgg_k= \dim \glob(\frakgg_k)=\dim \left\{ u\cdot\phi:  \phi\in \Val^{U(n)}_k\right\},$$
where the second equality follows from $\widehat{u}=\frac{2}{\pi} \mu_{2n-2,n-1}$. 

Suppose $0\leq k<n$ and $\phi\in\Val_k^{U(n)}$. We claim that $u\cdot \phi=0$ implies $\phi=0$. Indeed, since $u=4s-t^2$, the algebra of unitarily invariant valuations is not only generated by $t$ and $s$ but also by $t$ and $u$. Therefore $\phi$  can be expressed as a polynomial in $t$ and $u$ and from $u\cdot \phi=0$ we get a relation in $t$ and $u$. For degrees strictly less than $n+2$, however, there exists only one relation and this one is given by the Fu polynomial
$$f_{n+1}=\frac{1}{(n+1)(-2)^{n}}\sum_{q=0}^{\lfloor \frac{n+1}{2}\rfloor} (-1)^q\binom{k}{2q} t^{n-2q+1} u^q$$
Since $f_{n+1}$ is not a multiple of $u$, we get $\phi=0$. We conclude that 
$$\dim \frakgg_k=\dim\Val_k^{U(n)}$$
for $0\leq k<n$ and consequently $\frakgg_k=\spn\{\Gamma_{2n-k-2,q}\}$.

From the relation given by the Fu polynomial $f_{n+1}$, we deduce that $t^{n+1}$ can be written as $u\cdot \phi$ with some unitarily invariant valuation $\phi$. Since the algebra of unitarily invariant valuations is generated by $t$ and $u$, we conclude that the map $\phi\mapsto u\cdot \phi$ from $\Val_{k-2}^{U(n)}$ to $\Val_k^{U(n)}$ is surjective whenever $k>n$. This implies that

$$\frakgg_k=\dim\Val_{k}^{U(n)}-1$$
for $n\leq k<2n-1$ and therefore $\frakgg_k=\spn\{\Gamma_{2n-k-2,q}\}$. 
\end{proof}

\begin{proposition}\label{prop_span} The module of unitarily invariant area measures is generated by two elements. More precisely,
$$\Area^{U(n)}_{2n-k-1}=\frakbb_k \oplus \frakgg_{k-1}$$
for $1\leq k<n$. If $n\leq k<2n$, then
$$\Area^{U(n)}_{2n-k-1}=\frakbb_k + \frakgg_{k-1}$$
and $\frakbb_k \cap\frakgg_{k-1}$ is $1$-dimensional.
 
\end{proposition}
\begin{proof}
We will first show that for any $\Psi\in\Area_k^{U(n)}$, $0\leq k\leq 2n-3$, there exist $\Psi_1\in\Area_{k+1}^{U(n)}$ and $\Psi_2\in\Area_{k+2}^{U(n)}$ such 
that
\begin{equation}\label{eq_tsspan}
	\Psi=\hat t*\Psi_1+\hat s*\Psi_2.
\end{equation}
It will be sufficient to prove this for $\Psi=B_{k,q}$ and $\Psi=\Gamma_{k,q}$, since these measures constitute a basis.
We have already proved in Lemma~\ref{lem_bgdim} that the image of the map $\frakgg$ equals the span of the measures $\{\Gamma_{k,q}\}$.
This immediately implies  \eqref{eq_tsspan} for $\Psi=\Gamma_{k,q}$.

We turn now to the case $\Psi=B_{k,q}$. Clearly, since $k \leq 2n-3$, $B_{k+1,q}$ and $B_{k+2,q+1}$ exist if $B_{k,q}$ does and $\max\{0,k-n\}< q$. 
Then by Proposition~\ref{prop_taction} and Proposition~\ref{prop_saction} there are positive numbers $a_{ij}>0$ such 
\begin{align*}
	\hat t*B_{k+1,q} &= a_{11} B_{k,q-1}+a_{12}B_{k,q}\qquad\text{modulo}\ \spn\{\Gamma_{k,q}\},\\
   	\hat s*B_{k+2,q+1} &= a_{21} B_{k,q-1}+a_{22}B_{k,q}.
\end{align*}
It is not difficult to check that the matrix $(a_{ij})$ is non-singular and hence we find numbers $c_1$ and $c_2$ such that
$$B_{k,q}= c_1 \hat t*B_{k+1,q} +c_2\hat s*B_{k+2,q+1}\qquad\text{modulo}\  \spn\{\Gamma_{k,q}\}.$$
If $\max\{0,k-n\}=q$, then $B_{k+1,q+1}$ and $B_{k+2,q+2}$ exist if $B_{k,q}$ does. As before there is a non-singular matrix $(a_{ij})$ such that 
\begin{align*}
	\hat t*B_{k+1,q+1} &= a_{11} B_{k,q}+a_{12}B_{k,q+1}\qquad\text{modulo}\ \spn\{\Gamma_{k,q}\},\\
   	\hat s*B_{k+2,q+2} &= a_{21} B_{k,q}+a_{22}B_{k,q+1}
\end{align*}
and we deduce that there are numbers $c_1$ and $c_2$ such that
$$B_{k,q}= c_1 \hat t*B_{k+1,q+1} +c_2\hat s*B_{k+2,q+2}\qquad\text{modulo}\  \spn\{\Gamma_{k,q}\}.$$
Since we have already proved (\ref{eq_tsspan}) if $\Psi$ is a linear combination of some $\Gamma_{k,q}$, we conclude that (\ref{eq_tsspan}) also holds true for $\Psi=B_{k,q}$. 

By induction on $k$ we obtain from the above that 
$$\Area^{U(n)}_{2n-k-1}=\frakbb_k + \frakgg_{k-1}$$
for $1\leq k<2n$. From \eqref{eq_dimbb}, \eqref{eq_dimgg}, and Corollary~\ref{cor_dimarea}, we conclude that for dimensional reasons $\frakbb_k\cap \frakgg_{k-1}$ is $0$-dimensional if $1\leq k<n$ and $1$-dimensional if $n\leq k<2n$.
\end{proof}

\begin{lemma}\label{lem_kernel}

$$\frakbb(p_n)=\frakgg(q_{n-1})\qquad \text{and}\qquad \frakgg(p_n)=0$$

\end{lemma}
\begin{proof}
Since the image of $\frakgg$ coincides with the span of the $\Gamma_{k,q}$ and the globalization map is injective on this subspace, 
the second assertion follows immediately from $\widehat{u}=\frac{2}{\pi} \mu_{2n-2,n-1}$ and Proposition~\ref{lem_rel}. 

To prove the first assertion we first show that 

\begin{equation}\label{eq_pnzero} \frakbb(p_n) \in \spn\{\Gamma_{k,q}\}. \end{equation}
Suppose that $n$ is even,  $n=2m$.  Using Lemma~\ref{lem_magic}(2) with $j=m-i$ yields 
\begin{equation}\label{eq_evenpart}\frac{\pi^{m}}{4^{m}m!} \binom{2m+1}{2i+1} \frakbb( t^{2m-2i}u^i)\equiv \sum_{k=0}^{m-i} (2k+1)\binom{m-k}{i} B_{n-1,m-k-1}\end{equation}
modulo $\spn\{\Gamma_{k,q}\}$. Multiplying the right-hand side by $(-1)^i$ and summing over $i$, we obtain
\begin{align*}
	\sum_{i=0}^m (-1)^i \sum_{k=0}^{m-i} &(2k+1)\binom{m-k}{i} B_{n-1,m-k-1}\\
&= \sum_{\substack{0\leq i,k\leq m\\ k\leq m-i}}(-1)^i(2k+1)\binom{m-k}{i} B_{n-1,m-k-1} \\
&= \sum_{k=0}^m\left(\sum_{i=0}^{m-k} (-1)^i \binom{m-k}{i}\right)(2k+1)B_{n-1,m-k-1}\\
&=0.
\end{align*}

Suppose now that $n=2m+1$. We put 
$$p'_n=\frac{(-1)^n}{2^n}\sum_{i=0}^{m}  (-1)^i \binom{n+1}{2i+1}t^{2m-2i}u^i$$
such that $tp'_n=p_n$. Using Lemma~\ref{lem_magic}(2) with $j=m-i$, we see that modulo $\spn\{\Gamma_{k,q}\}$,
\begin{align*}
	\frakbb(p'_n) &\equiv \frac{(-1)^n(n+1)m!}{2 \pi^m} \sum_{i=0}^m\sum_{k=0}^{m-i} (-1)^i\frac{2k+1}{n-2i}\binom{m-k}{i} B_{n,m-k}\\
&= \frac{(-1)^n(n+1)m!}{2\pi^m}\sum_{k=0}^m\left(\sum_{i=0}^{m-k}(-1)^i \frac{2k+1}{n-2i}\binom{m-k}{i}\right) B_{n,m-k}\\
&=\frac{(-1)^n(n+1)m!}{2\pi^m} \sum_{q=0}^m\left(\sum_{i=0}^{q}(-1)^i \frac{2(m-q)+1}{2(m-i)+1}\binom{q}{i}\right) B_{n,q}\\
&=\frac{(-1)^n(n+1)m!}{2\pi^m}\sum_{q=0}^mc_q B_{n,q},
\end{align*}
where $$c_q= \sum_{i=0}^{q}(-1)^i \frac{2(m-q)+1}{2(m-i)+1}\binom{q}{i}.$$
Hence, using Proposition~\ref{prop_taction}, we arrive at
$$\frakbb(p_n)=\frakbb(tp'_n)\equiv \frac{(-1)^n(n+1)m!\omega_{n+1}}{2\pi^{m+1}\omega_n}\sum_{q=0}^{m-1} (n-2q-1)\left(c_{q+1}+ \frac{2(q+1)}{n-2q}c_q\right) B_{n-1,q}, $$
modulo $\spn\{\Gamma_{k,q}\}$. It is not difficult to see that for each $q$ the expressions in brackets vanish and thus \eqref{eq_pnzero} follows.

Since we know now that both $\frakbb(p_n)$ and $\frakgg(q_{n-1})$ are contained in the span of the $\Gamma_{k,q}$, we conclude that $\frakbb(p_n)$ and $\frakgg(q_{n-1})$ are equal if and only of their images under the globalization map coincide. 
Since
$$\FF(\glob(\frakbb(p_n)))=\frac{\pi}{2} t\cdot p_n $$
and 
$$\FF(\glob(\frakgg(q_{n-1})))=\frac{\pi}{2} u\cdot q_{n-1},$$
it is sufficient to prove $ (4s-t^2) q_{n-1}-t p_n=0$ locally at $n$. This, however, follows immediately from \eqref{eq_piqi}.
\end{proof}

We are now ready to complete the proof of Theorem~\ref{thm_main}.

\begin{proof}[Proof of Theorem~\ref{thm_main}]
Consider the module homomorphism $h\colon \Val^{U(n)}\oplus\Val^{U(n)}\rightarrow \Area^{U(n)}$ given by
$$(p, q)\mapsto \frakbb(p)+\frakgg(q).$$
By Proposition \ref{prop_span}, $h$ is surjective. We denote by $I_n\subset \Val^{U(n)}\oplus  \Val^{U(n)}$ the submodule generated by 
$(p_n,-q_{n-1})$ and $ (0,p_n)$.  
From Lemma~\ref{lem_kernel} we conclude that $I_n\subset \ker h$. 

Suppose now that $h(p,q)=0$. Then either $\frakbb(p)=\frakgg(q)=0$ or $\frakbb(p)=-\frakgg(q)\neq 0$. In the first case, we deduce from Proposition~\ref{lem_rel}, Lemma~\ref{lem_bgdim},
 and Lemma~\ref{lem_kernel} that $(p,q)\in I_n$. In the other case, $\frakbb(p)=-\frakgg(q)$ is a non-trivial element of the intersection of $\img \frakbb$ and $\img\frakgg$. 
From Proposition~\ref{prop_span} and Lemma~\ref{lem_kernel}, we deduce that $(p,q)\in I_n$. Hence $\ker h= I_n$.
\end{proof}

\subsection{Angular area measures}

Let $\Angular^{U(n)}\subset \Area^{U(n)}$ denote the subspace of angular, unitarily invariant area measures. The subspace of angular area measures has played an important role in our analysis; now we aim to give a description of it in terms of the module structure. We will see below that $\Angular^{U(n)}$ is not a submodule and that it is not invariant under the action of neither $\widehat{t}$, $\widehat{s}$, nor $\widehat{u}$.

\begin{lemma}\label{lem_bgangular}
Suppose $i>0$ and $j\geq 0$ are  integers. Then the area measure
\begin{equation}
\label{eq_bgangular}
(j+1)\frakbb(u^it^j)+2i(2i+j+2)\frakgg(u^{i-1}t^{j+1})\end{equation}
is angular. 
\end{lemma}
\begin{proof}
 The subspace of angular area measures coincides by Lemma~\ref{lem_deltaangular} with the span of the $\Delta_{k,q}$. Hence it is sufficient to prove that the measure \eqref{eq_bgangular} is an element of the latter subspace. As in the proof of Lemma~\ref{lem_magic}, one can check by induction that for $i>0$ and $m\geq 0$ there exist constants $c_{im}$ such that
\begin{align*}
c_{im} \frakbb(& t^{2m}u^i) = (2m+1)\sum^{\min(m,n-i-m-1)}_{k=0} \binom{i+m-k}{i} \Delta_{2(n-i-m)-1,n-i-m-k-1}\\
&-2(i+m+1) \sum^{\min(m,n-i-m-1)}_{k=0} (2k+1)\binom{m+i-1-k}{i-1} N_{2(n-i-m)-1,n-i-m-k-1}
\end{align*}
and 
\begin{align*}
c_{im} \frakgg(&t^{2m+1}u^{i-1}) = (2m+1)\sum^{\min(m,n-i-m-1)}_{k=0} \binom{i+m-k}{i} \Delta_{2(n-i-m)-1,n-i-m-k-1}\\
& + \frac{2m+1}{2i}\sum^{\min(m,n-i-m-1)}_{k=0} (2k+1)\binom{m+i-1-k}{i-1} N_{2(n-i-m)-1,n-i-m-k-1}
\end{align*}
The constants $c_{im}$ are explicitly given by 
$$c_{im}=\frac{(2i+1)\pi^{i+m}}{4^{i+m}(i+m)!}\binom{2i+2m+1}{2m}.$$
From this we deduce that \eqref{eq_bgangular} holds if $j=2m$. Using the above formulas for $\frakbb(t^{2m}u^i)$ and $\frakgg(t^{2m+1}u^{i-1})$, a simple, but long and tedious calculation shows that \eqref{eq_bgangular} holds true also for $j=2m+1$. 
\end{proof}

In the following we denote by $\RR[t,u]$ the graded polynomial algebra of the variables $t$ and $u$ with formal degrees $\deg t =1$ and $\deg u=2$. 

\begin{theorem}

\begin{enumerate}
\item The image of the map $A\colon \RR[t,u] \times \RR[u]	\rightarrow	\Area^{U(n)}$ given by
		$$A(p,q)= \frakbb\left(\left[t\frac{\partial }{\partial t} +1\right]p\right)+\frakgg\left( 2t\frac{\partial}{\partial u}\left[ t\frac{\partial}{\partial t} + 2u \frac{\partial}{\partial u}+2\right]p+q \right)$$
coincides with the subspace of angular area measures.

\item Suppose $p,q\in \RR[t,u]$. If 
\begin{equation}\label{eq_angularitycond}
\frac{\partial}{\partial u} \left[ t\frac{\partial}{\partial t} + 2u\frac{\partial}{\partial u} +2\right] p =\frac{1}{2} \frac{\partial q}{\partial t},
\end{equation}
then 
$$\frakbb(p)+ \frakgg(q) \in \Angular^{U(n)}\qquad \text{for every}\quad  n\geq1.$$
Conversely, if $$\frakbb(p)+ \frakgg(q) \in \Angular^{U(n)}$$ for some $n>\deg p, 1+ \deg q$, then \eqref{eq_angularitycond} holds.

\end{enumerate}

\end{theorem}

\begin{proof} 
Using Lemma~\ref{lem_bgangular} and Lemma~\ref{lem_magic}(1), we see that the image of $A$ is contained in the subspace of angular area measures.  To prove surjectivity, it is sufficient to show that the composed map $\glob\circ A$ maps onto $\bigoplus^{2n-1}_{k=0}\Val_k^{U(n)}$. This follows immediately from
\begin{align*}\glob\circ A(t^ju^i,0)& =\big((j+1)+2i(2i +j+2)\big) \widehat{t}^{j+1}\widehat{u}^i,\\
    \glob\circ A(0,u^i) &= \frac{\pi}{2} \widehat{u}^{i+1},
\end{align*}
and the fact that the algebra $(\Val^{U(n)}, \;*\;)$ is generated by $\widehat{t}$  and $\widehat{u}$.

If $p,q\in \RR[t,u]$ are such that \eqref{eq_angularitycond} holds, then Lemma~\ref{lem_bgangular} implies that $\frakbb(p)+ \frakgg(q) \in \Angular^{U(n)}$ for every $n\geq 1$. Conversely, suppose that $\frakbb(p)+ \frakgg(q) \in \Angular^{U(n)}$ for some $n$ with $n>\deg p, 1+\deg q   $. By the surjectivity of $A$ we find  $p_1\in \RR[t,u]$ and $p_2\in \RR[u]$ with $n>\deg p_1, 1+\deg q_2 $ such that
$$\frakbb(p)+ \frakgg(q)=A(p_1,p_2).$$
In particular, 
$$\frakbb\bigg(p - \left[t\frac{\partial }{\partial t} +1\right]p_1\bigg)=\frakgg\bigg(2t\frac{\partial}{\partial u}\left[ t\frac{\partial}{\partial t} + 2u \frac{\partial}{\partial u}+2\right]p_1+p_2-q\bigg).$$
From Proposition~\ref{prop_span}, we deduce that 
$$p=\left[t\frac{\partial }{\partial t} +1\right]p_1$$
and 
$$q = 2t\frac{\partial}{\partial u}\left[ t\frac{\partial}{\partial t} + 2u \frac{\partial}{\partial u}+2\right]p_1+p_2.$$
It is now easy to check that $p$ and $q $ satisfy \eqref{eq_angularitycond}.

\end{proof}

\subsection{Area and curvature measures}\label{sec_areacurv}

Using the explicit description of the module of unitarily invariant area measures, we try to shed some light on the relations between area and curvature measures. 

\begin{proposition}\label{prop_areacurv}
Let $\omega\in \Omega^{n-1}(SV)$ be translation-invariant and let $\Psi_\omega\in\Area$ and $\Phi_\omega\in\Curv$ denote the area and curvature measure represented by $\omega$. If $K\in\calK^{sm}$, then the Gauss map  $\nu:\partial K\rightarrow S(V)$ is a bijection and
$$\Psi_\omega(K)\ \text{ is the pushforward measure of}\ \Phi_\omega(K)$$
under the Gauss map. Moreover, the assignment $\Psi_\omega\mapsto \Phi_\omega$ is a well-defined, linear injection from $\Area$ into $\Curv$ with a $1$-dimensional cokernel.
\end{proposition}

\begin{proof}
We only prove that  $\Psi_\omega\mapsto \Phi_\omega$ is well-defined and injective; the rest is clear. If two forms $\omega_1$ and $\omega_2$ represent the same area measure $\Psi_{\omega_1}=\Psi_{\omega_2}$, then $\omega_1-\omega_2 $ is contained in the ideal $(\alpha, d\alpha)$. Hence also  $\Phi_{\omega_1}=\Phi_{\omega_2}$ which proves that the map is well-defined. Similarly, if $\Phi_\omega=0$, then $\omega\in (\alpha, d\alpha)$ and thus $\Psi_\omega=0$. 
\end{proof}

\begin{remark}
Using the explicit description of the modules of unitarily invariant area and curvature measures (see \cite{bernig_etal12}), one can show that if $n>1$, then there exists no injective linear map  
$$F\colon \Area^{U(n)}\rightarrow \Curv^{U(n)}$$ 
satisfying
$$F(\widehat\phi*\Psi)=\phi \cdot F(\Psi)$$
whenever $\phi\in\Val^{U(n)}$ and $\Psi \in\Area^{U(n)}$. The same conclusion holds if $\Curv^{U(n)}$ is replaced by
$\Curv^{U(n)}/(\Delta_{2n,n})$, $\Delta_{2n,n}(K,A)=\vol_{2n}(K\cap A)$. This suggests that a reasonable extension of Alesker's Fourier transform to a map from $\Area(V)$ into $\Curv(V)$ does not exist.
\end{remark}


\begin{bibdiv}
\begin{biblist}

\bib{abardia12}{article}{
   author={Abardia, J.},
   title={Difference bodies in complex vector spaces},
   journal={J. Funct. Anal.},
   volume={263},
   date={2012},
   number={11},
   pages={3588--3603},
}


\bib{abardia_bernig11}{article}{
   author={Abardia, J.},
   author={Bernig, A.},
   title={Projection bodies in complex vector spaces},
   journal={Adv. Math.},
   volume={227},
   date={2011},
   pages={830--846},
}


\bib{abardia_etal12}{article}{
   author={Abardia, J.},
   author={Gallego, E.},
   author={Solanes, G.},
   title={The Gauss-Bonnet theorem and Crofton-type formulas in complex
   space forms},
   journal={Israel J. Math.},
   volume={187},
   date={2012},
   pages={287--315},
}

\bib{aleksandrov37}{article}{
	title={Zur Theorie der gemischten Volumina von konvexen K\"orpern I.},
	author={Aleksandrov, A. D.},
	journal={Mat. Sbornik},
	language={Russian},
	volume={44},
	date={1937},
	pages={947--972},
}

\bib{alesker99}{article}{
   author={Alesker, S.},
   title={Continuous rotation invariant valuations on convex sets},
   journal={Ann. of Math. (2)},
   volume={149},
   date={1999},
   pages={977--1005},
}

\bib{alesker01}{article}{
    title={Description of translation invariant valuations on convex sets with solution of P. McMullen's conjecture},
    author={Alesker, S.},
    journal={Funct. Anal.},
    volume={11},
    date={2001},
    pages={244--272}
}

\bib{alesker03}{article}{
	title={Hard Lefschetz theorem for valuations, complex integral geometry, and unitarily invariant valuations},
	author={Alesker, S.},
	journal={J. Differential Geom.},
	volume={63},
	date={2003},
	pages={63--95}
}

\bib{alesker04}{article}{
	title={The multiplicative structure on continuous polynomial valuations}, 
	author={Alesker, S.},
	journal={Geom. Funct. Anal.},
	volume={14},
	date={2004},
	pages={1--26}
}

\bib{alesker06}{article}{
	title={Theory of valuations on manifolds I. Linear spaces},
	author={Alesker, S.},
	journal={Israel J. Math.},
	volume={156},
	date={2006},
	pages={311--339}
}

\bib{alesker06b}{article}{
   author={Alesker, S.},
   title={Theory of valuations on manifolds. II},
   journal={Adv. Math.},
   volume={207},
   date={2006},
   pages={420--454},
}

\bib{alesker07a}{article}{
   author={Alesker, S.},
   title={Theory of valuations on manifolds. IV. New properties of the
   multiplicative structure},
   conference={
      title={Geometric aspects of functional analysis},
   },
   book={
      series={Lecture Notes in Math.},
      volume={1910},
      publisher={Springer},
      place={Berlin},
   },
   date={2007},
   pages={1--44},
}


\bib{alesker07b}{article}{
   author={Alesker, S.},
   title={Theory of valuations on manifolds: a survey},
   journal={Geom. Funct. Anal.},
   volume={17},
   date={2007},
   pages={1321--1341},
}
		
\bib{alesker11}{article}{
	title={A Fourier-type transform on translation-invariant valuations on convex sets},
	author={Alesker, S.},
	journal={Israel J. Math.},
	volume={181},
	date={2011},
	pages={189--294}
}


\bib{alesker_fu08}{article}{
   author={Alesker, S.},
   author={Fu, J. H. G.},
   title={Theory of valuations on manifolds. III. Multiplicative structure
   in the general case},
   journal={Trans. Amer. Math. Soc.},
   volume={360},
   date={2008},

   pages={1951--1981},
}
 
 

\bib{alesker_etal}{article}{
	author={Alesker, S.},
	author={Bernig, A.},
	author={Schuster, F. E.},
   title={Harmonic analysis of translation invariant valuations},
   journal={Geom. Funct. Anal.},
   volume={21},
   date={2011},
   pages={751--773},
}
	
\bib{bernig09}{article}{
   author={Bernig, A.},
   title={A Hadwiger-type theorem for the special unitary group},
   journal={Geom. Funct. Anal.},
   volume={19},
   date={2009},
   pages={356--372},
}


\bib{bernig11}{article}{
   author={Bernig, A.},
   title={Integral geometry under $G_2$ and ${\rm Spin}(7)$},
   journal={Israel J. Math.},
   volume={184},
   date={2011},
   pages={301--316},
}


\bib{bernig12a}{article}{
	title={Algebraic integral geometry},
	author={Bernig, A.},
	book={
			title={Global Differential Geometry},	
			editor={B\"ar, C.},
			editor={Lohkamp, J.},
			editor={Schwarz, M.},
			publisher={Springer},
			date={2012}
			},			
}




\bib{bernig_broecker07}{article}{
	title={Valuations on manifolds and Rumin cohomology},
	author={Bernig, A.},
	author={Br\"ocker, L.},
	journal={J. Differential Geom.},
	volume={75},
	date={2007},
	pages={433--457}
} 



\bib{bernig_fu06}{article}{
	title={Convolution of convex valuations},
	author={Bernig, A.},
	author={Fu, J. H. G.},
	journal={Geom. Dedicata},
	volume={123},
	date={2006},
	pages={153--169}
}




\bib{bernig_fu11}{article}{
	title={Hermitian integral geometry},
	author={Bernig, A.},
	author={Fu, J. H. G.},
	journal={Ann. of Math. (2)},
	volume={173},
	date={2011},
	pages={907--945}
}


\bib{bernig_etal12}{article}{
	title={Integral geometry of complex space forms},
	author={Bernig, A.},
	author={Fu, J. H. G.},
	author={Solanes, G.},
	eprint={arXiv:1204.0604v1 [math.DG]},
} 	

\bib{blair10}{book}{
   author={Blair, D. E.},
   title={Riemannian geometry of contact and symplectic manifolds},
   series={Progress in Mathematics},
   edition={2},
   publisher={Birkh\"auser Boston Inc.},
   place={Boston, MA},
   date={2010},
}

\bib{blaschke55}{book}{
   author={Blaschke, W.},
   title={Vorlesungen \"uber Integralgeometrie},
   language={German},
   publisher={Deutscher Verlag der Wissenschaften},
   place={Berlin},
   date={1955},
}

\bib{broecker_dieck85}{book}{
   author={Br{\"o}cker, T.},
   author={tom Dieck, T.},
   title={Representations of compact Lie groups},
   series={Graduate Texts in Mathematics},
   volume={98},
   publisher={Springer-Verlag},
   place={New York},
   date={1985},
}

\bib{chern52}{article}{
   author={Chern, S.-S.},
   title={On the kinematic formula in the Euclidean space of $n$ dimensions},
   journal={Amer. J. Math.},
   volume={74},
   date={1952},
   pages={227--236},

}


\bib{federer59}{article}{
	title={Curvature measures},
	author={Federer, H.},
	journal={Trans. Amer. Math. Soc.},
	volume={93},
	date={1959},
	pages={418--491}
}

\bib{federer69}{book}{
   author={Federer, H.},
   title={Geometric measure theory},
   series={Die Grundlehren der mathematischen Wissenschaften},
   publisher={Springer-Verlag New York Inc., New York},
   date={1969},
   pages={xiv+676},
}

\bib{fu90}{article}{
	title={Kinematic formulas in integral geometry},
	author={Fu, J. H. G.},
	journal={Indiana Univ. Math. J.},
	volume={39},
	date={1990},
	pages={1115--1154}
}

\bib{fu94}{article}{
   author={Fu, J. H. G.},
   title={Curvature measures of subanalytic sets},
   journal={Amer. J. Math.},
   volume={116},
   date={1994},
   pages={819--880},
}
	


\bib{fu06}{article}{
	title={Structure of the unitary valuation algebra},
	author={Fu, J. H. G.},
	journal={J. Differential Geom.},
	volume={72},
	date={2006},
	pages={509--533}
}


\bib{fu11}{article}{
	title={Algebraic integral geometry},
	author={Fu, J. H. G.},
	eprint={	arXiv:1103.6256v1 [math.DG]},
}




\bib{haberl11}{article}{
   author={Haberl, C.},
   title={Blaschke valuations},
   journal={Amer. J. Math.},
   volume={133},
   date={2011},
   number={3},
   pages={717--751},
}

\bib{haberl_schuster09}{article}{
   author={Haberl, C.},
   author={Schuster, F. E.},
   title={General $L_p$ affine isoperimetric inequalities},
   journal={J. Differential Geom.},
   volume={83},
   date={2009},
   pages={1--26},
}

\bib{humphreys72}{book}{
   author={Humphreys, J. E.},
   title={Introduction to Lie algebras and representation theory},
   series={Graduate Texts in Mathematics},
   volume={9},
   publisher={Springer-Verlag},
   place={New York},
   date={1972},
}
	
\bib{klain_rota97}{book}{
   author={Klain, D. A.},
   author={Rota, G.-C.},
   title={Introduction to geometric probability},
   series={Lezioni Lincee},
   publisher={Cambridge University Press},
   place={Cambridge},
   date={1997},
}
	



\bib{klain00}{article}{
	title={Even valuations on convex bodies},
	author={Klain, D. A.},
	journal={Trans. Amer. Math. Soc.},
	volume={352},
	date={2000},
	pages={71--93}
}

\bib{ludwig02}{article}{
   author={Ludwig, M.},
   title={Projection bodies and valuations},
   journal={Adv. Math.},
   volume={172},
   date={2002},
   pages={158--168},
}

\bib{ludwig03}{article}{
   author={Ludwig, M.},
   title={Ellipsoids and matrix-valued valuations},
   journal={Duke Math. J.},
   volume={119},
   date={2003},
   pages={159--188},
}
		
\bib{ludwig06}{article}{
   author={Ludwig, M.},
   title={Intersection bodies and valuations},
   journal={Amer. J. Math.},
   volume={128},
   date={2006},
   pages={1409--1428},
}

\bib{ludwig10}{article}{
   author={Ludwig, M.},
   title={Minkowski areas and valuations},
   journal={J. Differential Geom.},
   volume={86},
   date={2010},
   pages={133--161},
}

\bib{ludwig_reitzner10}{article}{
   author={Ludwig, M.},
   author={Reitzner, M.},
   title={A classification of ${\rm SL}(n)$ invariant valuations},
   journal={Ann. of Math. (2)},
   volume={172},
   date={2010},
   pages={1219--1267},
}

\bib{lutwak_etal00}{article}{
   author={Lutwak, E.},
   author={Yang, D.},
   author={Zhang, G.},
   title={$L_p$ affine isoperimetric inequalities},
   journal={J. Differential Geom.},
   volume={56},
   date={2000},
   pages={111--132},
}

\bib{lutwak_etal10}{article}{
   author={Lutwak, E.},
   author={Yang, D.},
   author={Zhang, G.},
   title={Orlicz centroid bodies},
   journal={J. Differential Geom.},
   volume={84},
   date={2010},
   pages={365--387},
}

\bib{mcmullen77}{article}{
   author={McMullen, P.},
   title={Valuations and Euler-type relations on certain classes of convex
   polytopes},
   journal={Proc. London Math. Soc. (3)},
   volume={35},
   date={1977},
   pages={113--135},
}



\bib{nijenhuis74}{article}{
   author={Nijenhuis, A.},
   title={On Chern's kinematic formula in integral geometry},
   journal={J. Differential Geom.},
   volume={9},
   date={1974},
   pages={475--482},
}
	



\bib{parapatits12}{article}{
	title={$SL(n)$-contravariant $L_p$-Minkowski valuations},
	author={Parapatits, L.},
	journal={to appear in Trans. Amer. Math. Soc.},
}


\bib{parapatits_schuster12}{article}{
	title={The Steiner formula for Minkowski valuations},
	author={Parapatits, L.},
	author={Schuster, F. E.},
	journal={Adv. Math.},
	volume={230},
	date={2012},
	pages={978--994},
}


\bib{parapatits_wannerer12}{article}{
	title={On the inverse Klain map},
	author={Parapatits, L.},
	author={Wannerer, T.},
	journal={to appear in Duke Math. J.},
}

\bib{rumin94}{article}{
	title={Formes diff\'{e}rentielles sur les vari\'{e}t\'{e}s de contact},
	author={Rumin, M.},
	journal={J. Differential Geom.},
	volume={39},
	date={1994},
	pages={281--330}
}

\bib{schneider71}{article}{
	title={On Steiner points of convex bodies},
	author={Schneider, R.},
	journal={Israel J. Math.},
	volume={9},
	date={1971},
	pages={241--249}
}

\bib{schneider75}{article}{
   author={Schneider, R.},
   title={Kinematische Ber\"uhrma\ss e f\"ur konvexe K\"orper und
   Integralrelationen f\"ur Oberfl\"achenma\ss e},
   journal={Math. Ann.},
   volume={218},
   date={1975},
   pages={253--267},
   language={German},
}


 

\bib{schneider_book}{book}{
	title={Convex bodies: The Brunn-Minkowski theory},
	author={Schneider, R.},
	publisher={Cambridge University Press},
	address={Cambridge},
	date={1993}
}


\bib{schuster08}{article}{
   author={Schuster, F. E.},
   title={Valuations and Busemann-Petty type problems},
   journal={Adv. Math.},
   volume={219},
   date={2008},
   pages={344--368},
}


\bib{schuster10}{article}{
   author={Schuster, F. E.},
   title={Crofton measures and Minkowski valuations},
   journal={Duke Math. J.},
   volume={154},
   date={2010},
   pages={1--30},
}

	



\bib{takeuchi94}{book}{
	title={Modern spherical functions},
	author={Takeuchi, M.},
	publisher={American Mathematical Society},
	address={Providenc, RI},
	date={1994},

}


\bib{wannerer12}{article}{
   author={Wannerer, T.},
   title={${\rm GL}(n)$ covariant Minkowski valuations},
   journal={Indiana Univ. Math. J.},
   volume={60}, 
   date={2011},
   pages={1655--1672}, 
}
	


\bib{wilf94}{book}{
   author={Wilf, H. S.},
   title={generatingfunctionology},
   edition={2},
   publisher={Academic Press Inc.},
   place={Boston, MA},
   date={1994},
   pages={x+228},
}


\end{biblist}
\end{bibdiv}
\end{document}